\documentclass[11pt]{article}
\usepackage[letterpaper,margin=1in]{geometry}
\usepackage{amsmath,amssymb,amsthm,mathtools}
\usepackage{booktabs}
\usepackage{tabularx,array}
\usepackage{enumitem}
\usepackage{microtype}
\usepackage{needspace}
\usepackage{algorithm}
\usepackage{algpseudocode}
\usepackage{bm}
\usepackage{graphicx}
\usepackage{placeins}
\usepackage{float}
\usepackage{xurl}
\usepackage[hidelinks]{hyperref}

\DeclareMathSizes{6.6}{6.6}{5}{5}

\hypersetup{
  pdftitle={A Fixed-Penalty Linearized Augmented Lagrangian Method with Classical Multiplier Updates},
  pdfauthor={Benqi Liu, Kangkang Deng, Zichen Wang, and Zaiwen Wen}
}

\newtheorem{theorem}{Theorem}[section]
\newtheorem{lemma}[theorem]{Lemma}
\newtheorem{proposition}[theorem]{Proposition}
\newtheorem{corollary}[theorem]{Corollary}
\newtheorem{assumption}[theorem]{Assumption}
\theoremstyle{definition}
\newtheorem{definition}[theorem]{Definition}
\theoremstyle{remark}

\numberwithin{equation}{section}
\numberwithin{figure}{section}
\numberwithin{table}{section}
\numberwithin{algorithm}{section}

\newenvironment{keywords}{%
  \par\medskip\noindent\textbf{Keywords.}\ }{\par}
\newenvironment{MSCcodes}{%
  \par\smallskip\noindent\textbf{Mathematics Subject Classification.}\ }{\par\medskip}

\newcommand{\R}{\mathbb{R}}
\newcommand{\E}{\mathbb{E}}
\newcommand{\cL}{\mathcal{L}}
\newcommand{\cO}{\mathcal{O}}
\newcommand{\eps}{\varepsilon}
\newcommand{\norm}[1]{\left\lVert #1\right\rVert}
\newcommand{\ip}[2]{\left\langle #1,#2\right\rangle}
\newcommand{\grad}{\nabla}
\newcommand{\Proj}{\operatorname{Proj}}

\newcommand{\leanLinkText}[1]{%
  \underline{\textbf{#1}}%
}
\newcommand{\leanDeclTag}[2]{%
  {\footnotesize\texttt{[Lean source, }}%
  \href{https://github.com/optpku/ReasBook/blob/v4.32.2/ReasBook/Papers/TR_LALM_theory/#1?plain=1\#L#2}%
    {\leanLinkText{\footnotesize\texttt{L#2}}}%
}
\newcommand{\leanDeclTagExtra}[2]{%
  \unskip
  {\footnotesize\texttt{, }}%
  \href{https://github.com/optpku/ReasBook/blob/v4.32.2/ReasBook/Papers/TR_LALM_theory/#1?plain=1\#L#2}%
    {\leanLinkText{\footnotesize\texttt{L#2}}}\allowbreak
}
\newcommand{\leanDeclTagEnd}{%
  \unskip
  {\footnotesize\texttt{]}}\enspace
}

\title{A Fixed-Penalty Linearized Augmented Lagrangian Method with Classical Multiplier Updates}
\author{%
Benqi Liu$^{1}$ \qquad Kangkang Deng$^{2}$ \qquad
Zichen Wang$^{3}$ \qquad Zaiwen Wen$^{1,*}$\\[0.7em]
\small $^{1}$Beijing International Center for Mathematical Research,
Peking University, Beijing 100871, China\\[-0.1em]
\small \href{mailto:bqliu@pku.edu.cn}{\texttt{bqliu@pku.edu.cn}},
\href{mailto:wenzw@pku.edu.cn}{\texttt{wenzw@pku.edu.cn}}\\
\small $^{2}$College of Science, National University of Defense Technology,
Changsha 410073, China\\[-0.1em]
\small \href{mailto:freedeng1208@gmail.com}{\texttt{freedeng1208@gmail.com}}\\
\small $^{3}$School of Mathematical Sciences, Peking University,
Beijing 100871, China\\[-0.1em]
\small \href{mailto:zichenwang25@stu.pku.edu.cn}{\texttt{zichenwang25@stu.pku.edu.cn}}\\
\small $^{*}$Corresponding author: Zaiwen Wen
}
\date{}

\begin{document}
\maketitle

\begin{abstract}
Augmented Lagrangian methods are effective for nonlinear equality-constrained
optimization, but solving their nonlinear primal subproblems can be expensive.
For smooth nonconvex problems with deterministic or stochastic objectives, we
propose a nonlinear-residual linearized augmented Lagrangian method (NR-LALM)
that replaces this subproblem by a regularized Gauss--Newton-type step while
retaining the classical multiplier update based on the nonlinear constraint
residual.  The resulting step is computed from one symmetric positive-definite
linear system, but the mismatch between the linearized primal model and the
nonlinear-residual update produces a quadratic constraint-linearization error
in the multiplier identity.  We show that this error can be controlled under
local regularity; multiplier boundedness and trajectory localization are
derived rather than assumed.  With fixed, accuracy-independent parameters,
deterministic NR-LALM finds an $\varepsilon$-approximate
Karush--Kuhn--Tucker (KKT) pair in $\mathcal O(\varepsilon^{-2})$ iterations
and first-order oracle evaluations.  For stochastic objectives, a projected
stochastic path-integrated differential estimator with safeguarded restarts
requires, in expectation, $\mathcal O(\varepsilon^{-3})$ stochastic-gradient
evaluations and $\mathcal O(\varepsilon^{-2})$ constraint and Jacobian
evaluations.  Compactness and a Kurdyka--{\L}ojasiewicz condition further yield
finite-length convergence of the deterministic primal--dual sequence.  An
optional minimum-norm second-order correction reduces the
constraint-linearization error from second to fourth order without changing
the complexity orders. All theoretical results are formalized in Lean~4.  Numerical experiments confirm the predicted error
orders and show favorable performance on high-dimensional deterministic and
stochastic problems.  
\end{abstract}



\section{Introduction}
\label{sec:introduction}

We study the nonlinear equality-constrained problem
\begin{equation}
    \min_{x\in\R^n} f(x)
    \qquad \text{subject to}\qquad c(x)=0,
    \label{eq:problem}
\end{equation}
where $f:\R^n\to\R$ may be nonconvex and $c:\R^n\to\R^m$ may be nonlinear; we also allow stochastic objectives $f(x)=\E[F(x;\xi)]$ with deterministic constraints.  We write $\grad c(x):=(\grad c_1(x),\ldots,\grad c_m(x))\in\R^{n\times m}$. Introduced by Hestenes and Powell in 1969~\cite{Hestenes1969,Powell1969}, the augmented Lagrangian method (ALM) has become a standard framework for constrained optimization~\cite{Bertsekas1982,NocedalWright2006,Rockafellar1973}.
For $\rho>0$, ALM approximately minimizes the augmented Lagrangian before applying $\lambda_{k+1}=\lambda_k+\rho c(x_{k+1})$, but its primal subproblem is generally nonlinear and expensive.  Linearizing the smooth terms with a quadratic proximal term reduces it to one symmetric positive-definite system.  With $p_k:=x_{k+1}-x_k$, local Lipschitz continuity of $\grad c$ gives
$c(x_{k+1})=c(x_k)+\grad c(x_k)^\top p_k+\cO(\norm{p_k}^2)$.
The classical update therefore yields a quadratically perturbed multiplier identity rather than the exact identity of a linearized-residual update. Alacaoglu and Wright~\cite{AlacaogluWright2024} asked whether deterministic ALM for nonlinearly constrained nonconvex problems can attain the best-known $\cO(\eps^{-3})$ complexity with a fixed, accuracy-independent penalty and
large dual steps.  We ask whether this perturbation can be controlled without assumed multiplier or trajectory bounds while retaining complexity guarantees in the deterministic and stochastic settings.

\subsection{A linearized ALM with nonlinear-residual updates}

We propose the nonlinear-residual linearized augmented Lagrangian method (NR-LALM).  For fixed $\rho,\beta>0$, its deterministic iteration is
$$
\begin{aligned}
\bigl(\beta I_n+\rho\grad c(x_k)\grad c(x_k)^\top\bigr)p_k
&=
-\grad f(x_k)-\grad c(x_k)\bigl(\lambda_k+\rho c(x_k)\bigr),\\
x_{k+1}&=x_k+p_k,
\qquad
\lambda_{k+1}=\lambda_k+\rho c(x_{k+1}).
\end{aligned}
$$
The coefficient matrix is symmetric positive definite (SPD), and each iteration uses one objective gradient, one constraint Jacobian, one new constraint residual, and one SPD solve, with no nonlinear inner subproblem. Let $d_k:=c(x_{k+1})-c(x_k)-\grad c(x_k)^\top p_k$.  Local Lipschitz continuity gives $\norm{d_k}=\cO(\norm{p_k}^2)$, and the primal equation yields
$\grad f(x_k)+\grad c(x_k)\lambda_{k+1}+\beta p_k=\rho\grad c(x_k)d_k$. Thus the nonlinear-residual update creates a quadratic perturbation that must be controlled without assuming multiplier or trajectory boundedness.  Our contributions are as follows.
\begin{enumerate}[leftmargin=1.8em,labelsep=0.5em,itemsep=0em,topsep=0.1em]
    \item We give an affirmative answer, within the linearized augmented Lagrangian framework, to the open question raised by Alacaoglu and Wright~\cite{AlacaogluWright2024} concerning a fixed, accuracy-independent penalty and large dual steps.  For smooth nonconvex problems with nonlinear equality constraints, deterministic NR-LALM retains the classical multiplier update and finds an $\varepsilon$-approximate Karush--Kuhn--Tucker (KKT) pair in $\cO(\varepsilon^{-2})$ iterations and first-order oracle calls.  This improves on the $\cO(\varepsilon^{-3})$ target in that question.

    \item We derive, rather than assume, multiplier boundedness and trajectory localization.  Under local smoothness, a uniform linear independence constraint qualification (LICQ), and the stated parameter and localization conditions, we control the quadratic perturbation in the multiplier identity and establish Lyapunov descent, from which both conclusions follow.  Under compactness and a Kurdyka--{\L}ojasiewicz (KL) condition, we further show that the full primal--dual sequence has finite length and converges to a KKT pair.

    \item For stochastic objectives, a projected stochastic path-integrated differential estimator (SPIDER) with safeguarded restarts yields expected $\cO(\varepsilon^{-3})$ stochastic-gradient complexity and $\cO(\varepsilon^{-2})$ constraint/Jacobian complexity.  We also introduce an optional minimum-norm second-order correction (SOC), which reduces the
    constraint linearization error from second to fourth order.  Replacing the base constants by their corrected counterparts extends the deterministic and stochastic results to NR-LALM+SOC; its sufficient parameter region can
    strictly contain that of NR-LALM.

    \item We formally verify all theoretical results in the manuscript using Lean~4 and manually check that each formal statement faithfully represents its natural-language counterpart.  The complete Lean development is available in the public repository%
     \footnote{\href{https://github.com/optpku/ReasBook/tree/v4.32.2/ReasBook/Papers/TR_LALM_theory}%
      {\nolinkurl{https://github.com/optpku/ReasBook/tree/v4.32.2/ReasBook/Papers/TR_LALM_theory}}} and can also be checked through the ReasLab project%
      \footnote{\href{https://reaslab.io/share/w58YAU-8Rh-x_M7Xshn5WwR6df839.MTc.YWxs}%
      {\nolinkurl{https://reaslab.io/share/w58YAU-8Rh-x_M7Xshn5WwR6df839.MTc.YWxs}}}. For principal results, the manuscript includes direct links to the corresponding Lean declarations.
\end{enumerate}

\subsection{Related work}

Deterministic nonlinear ALM analyses include inexact frameworks
\cite{SahinEtAl2019,GrapigliaYuan2021,BirginMartinez2020} and the proximal ALM (ProxAL)~\cite{XieWright2021}, all of which retain nonlinear primal subproblems.  The linearized augmented Lagrangian method (L-AL)~\cite{BourkhissiNecoara2025} is closest: its fully linearized model uses backtracking, dynamic regularization, and a linearized-residual update. Perturbed linearized ALM employs a
scaled dual update for nonsmooth composite objectives~\cite{ElBourkhissiEtAl2025}, while Zhu and Li~\cite{ZhuLi2026} study related first-order schemes under local
regularity.  NR-LALM retains the classical update based on
$c(x_{k+1})$ and therefore produces the quadratic perturbation analyzed in Section~\ref{sec:deterministic}.

For stochastic objectives, the single-loop quadratic penalty method (QPM) uses no multiplier recursion and an increasing penalty, attaining $\widetilde{\cO}(\eps^{-4})$ sample complexity~\cite{AlacaogluWright2024}. The momentum-based linearized augmented Lagrangian method (MLALM) combines
recursive momentum with a nonlinear-residual update~\cite{ShiWangWang2025}, whereas the Fletcher stochastic ALM uses a multiplier map and tangential--normal steps~\cite{CuiEtAl2026}.  Further alternatives include
variance-reduced penalties under a global error bound~\cite{LuMeiXiao2026} and stochastic sequential quadratic programming~\cite{BerahasEtAl2021,CurtisONeillRobinson2024}.
Table~\ref{tab:intro-comparison} summarizes the differences;
$\widetilde{\cO}$ hides logarithmic factors, D and S denote deterministic and stochastic settings, CG denotes conjugate gradient, and stochastic bounds are in expectation.

\begin{table}[H]
\caption{Comparison with representative methods.}
\label{tab:intro-comparison}
\centering
\fontsize{6.8pt}{7.0pt}\selectfont
\setlength{\tabcolsep}{1.1pt}
\renewcommand{\arraystretch}{0.72}
\begin{tabular}{@{}>{\raggedright\arraybackslash}p{0.125\textwidth}
>{\raggedright\arraybackslash}p{0.165\textwidth}
>{\centering\arraybackslash}p{0.10\textwidth}
>{\raggedright\arraybackslash}p{0.115\textwidth}
>{\raggedright\arraybackslash}p{0.23\textwidth}
>{\raggedright\arraybackslash}p{0.225\textwidth}@{}}
\toprule
Method & Primal update & Dual update & Penalty & Complexity & Main assumptions \\
\midrule
ProxAL~\cite{XieWright2021}
& Inexact nonlinear solve
& Nonlinear $c(x_{k+1})$
& $\rho=\Omega(\eps^{-\eta})$, $\eta\in[0,2]$
& D: $\cO(\eps^{\eta-2})$ outer;
  $\cO(\eps^{-11/2})$ Newton--CG ($\eta=1$)
& Compact penalized level sets; uniform LICQ
\\ \hline
L-AL~\cite{BourkhissiNecoara2025}
& One SPD solve per trial
& Linearized residual
& Fixed $\rho$
& D: $\cO(\sqrt{\rho}\,\eps^{-2})$ Jacobian evaluations
& Compact penalized level sets; uniform LICQ
\\ \hline
Single-loop QPM~\cite{AlacaogluWright2024}
& Projected stochastic-gradient step
& None
& Increasing $\rho_k\asymp k^{1/4}$
& S: $\widetilde{\cO}(\eps^{-4})$
& Mean-squared smoothness; bounded data; uniform regularity
\\ \hline
MLALM~\cite{ShiWangWang2025}
& Stochastic proximal-gradient step
& Nonlinear $c(x_{k+1})$
& Runwise fixed; horizon-dependent
& S: $\cO(\eps^{-4})$; $\cO(\eps^{-3})$ with
  $\sqrt\eps$-feasible start and $\cO(\eps)$ initial error
& Mean-squared smoothness; constraint qualification
\\ \hline
Variance-reduced penalty~\cite{LuMeiXiao2026}
& Projected stochastic penalty-gradient step
& None
& Increasing $\rho_k$
& S: $\widetilde{\cO}(\eps^{-3})$ (recursive) or
  $\widetilde{\cO}(\eps^{-4})$ (Polyak), $\theta=1$
& Global error bound ($\theta\ge1$); average or Lipschitz smoothness
\\ \hline
Fletcher stochastic ALM~\cite{CuiEtAl2026}
& Stochastic tangential--normal step
& Map $\lambda(x)$; no recursion
& Adaptive merit parameter
& S: $\cO(\eps^{-3})$ total stochastic-oracle calls
& Mean-squared smoothness; Lipschitz Hessians; strong LICQ
\\ \hline
NR-LALM (this paper)
& One SPD solve per iteration
& Nonlinear $c(x_{k+1})$
& Fixed; accuracy-independent
& D: $\cO(\eps^{-2})$ iterations and first-order calls;
  S: $\cO(\eps^{-3})$ gradients and
  $\cO(\eps^{-2})$ constraint/Jacobian calls
& Local smoothness/uniform LICQ; regularity-region buffer and initialization;
  mean-squared smoothness; exact-set restart safeguard
\\
\bottomrule
\end{tabular}
\end{table}

\subsection{Organization}

Section~\ref{sec:deterministic} develops the deterministic analysis, and Section~\ref{sec:stochastic} treats the stochastic method.  The optional correction is studied in Section~\ref{sec:correction}; numerical results and
conclusions appear in Sections~\ref{sec:numerics} and~\ref{sec:conclusion}.

\section{Deterministic linearized augmented Lagrangian method}
\label{sec:deterministic}

This section develops the deterministic analysis of NR-LALM.  We first define approximate stationarity, state the local regularity assumptions, and present the algorithm and its preliminary estimates.  We then establish Lyapunov descent and trajectory localization, derive the $\cO(\varepsilon^{-2})$ complexity bound, and prove full-sequence convergence under compactness and a KL condition.

\subsection{Stationarity and local regularity assumptions}

Unless otherwise stated, vector norms are Euclidean and matrix norms are spectral.  For $\rho>0$ and $\lambda\in\R^m$, define
$\cL_\rho(x,\lambda):=f(x)+\ip{\lambda}{c(x)} +(\rho/2)\norm{c(x)}^2$. The following assumptions are imposed on an open set \(\mathcal U\).  We do not assume that the generated iterates lie in this set; that property is proved in Theorem~\ref{thm:localization}.

\begin{assumption}
\label{ass:smooth}
\leanDeclTag{Assumption_2_1/Regularity.lean}{34}
\leanDeclTagEnd
There exist a nonempty open set $\mathcal U\subseteq\R^n$ on which \(f\) and
\(c\) are continuously differentiable, a scalar $\underline{f}\in\R$, and
positive constants $G,L_f,M,L_c,\sigma$ with the following properties:
\begin{enumerate}[label=\textup{(\roman*)},leftmargin=2.5em,itemsep=0.15em,topsep=0.25em]
\item \(f(x)\ge \underline{f}\), \(\norm{\grad f(x)}\le G\), and
      \(\norm{\grad f(x)-\grad f(y)}\le L_f\norm{x-y}\) for all
      \(x,y\in\mathcal U\);
\item \(\norm{\grad c(x)}\le M\) and
      \(\norm{\grad c(x)-\grad c(y)}\le L_c\norm{x-y}\) for all
      \(x,y\in\mathcal U\);
\item $\norm{\grad c(x)u}\ge\sigma\norm{u}$ for all
      $x\in\mathcal U$ and $u\in\R^m$.
\end{enumerate}
\end{assumption}

Condition~(iii) is uniform LICQ on $\mathcal U$ and, in particular,
implies $m\le n$. The first-order KKT conditions are \(\grad f(x)+\grad c(x)\lambda=0\) and \(c(x)=0\). Since approximate stationarity depends on the multiplier, we distinguish
a primal point that admits an approximate multiplier from an explicitly specified primal--dual pair, following standard terminology for nonlinear equality constraints~\cite{XieWright2021,BourkhissiNecoara2025}.

\begin{definition}
\label{def:eps-kkt}
\leanDeclTag{Definition_2_2/KKT.lean}{31}
\leanDeclTagEnd
Given \(\varepsilon\ge0\), a point \(x\) is an
\(\varepsilon\)-KKT point of problem~\eqref{eq:problem} if there exists
\(\lambda\in\mathbb{R}^m\) such that
$$
\norm{\grad f(x)+\grad c(x)\lambda}\le\varepsilon,
\qquad
\norm{c(x)}\le\varepsilon.
$$
A specified pair \((x,\lambda)\) satisfying these inequalities is called
an \(\varepsilon\)-KKT pair.  For \(\varepsilon=0\), we call \(x\) a
KKT point and \((x,\lambda)\) a KKT pair.
\end{definition}

For subsequent estimates, define the aggregate KKT residual
\begin{equation}
    \mathcal R(x,\lambda)
    :=\left(\norm{\grad f(x)+\grad c(x)\lambda}^2
       +\norm{c(x)}^2\right)^{1/2}.
    \label{eq:kkt-residual}
\end{equation}
Hence \(\mathcal R(x,\lambda)\le\varepsilon\) implies that
\((x,\lambda)\) is an \(\varepsilon\)-KKT pair.  Conversely, every
\(\varepsilon\)-KKT pair satisfies
\(\mathcal R(x,\lambda)\le\sqrt{2}\,\varepsilon\); this fixed factor does not
affect any complexity order.
\subsection{Algorithm and preliminary estimates}
\label{sec:method}

We now present the deterministic NR-LALM iteration.  For fixed parameters $\rho,\beta>0$, the primal step $p_k$ is obtained by solving the linearized augmented-Lagrangian subproblem
$$
\min_{p\in\mathbb{R}^n} \ip{\grad f(x_k)}{p}
 +\ip{\lambda_k}{c(x_k) +\grad c(x_k)^\top p}
 +\frac{\rho}{2}\norm{c(x_k) +\grad c(x_k)^\top p}^2
 +\frac{\beta}{2}\norm{p}^2.
$$
Because $\beta>0$, this subproblem is strongly convex and hence has a unique minimizer.  Algorithm~\ref{alg:lalm} combines this primal step with the classical multiplier update based on the nonlinear constraint residual
$c(x_{k+1})$.
\begin{algorithm}[H]
\caption{Fixed-penalty NR-LALM}
\label{alg:lalm}
\leanDeclTag{Algorithm_2_1/Iteration.lean}{143}
\leanDeclTagEnd
\begin{algorithmic}[1]
\Require $x_0,\lambda_0$, fixed $\rho>0$, and fixed proximal coefficient $\beta>0$.
\For{$k=0,1,2,\ldots$}
    \State Compute the linearized augmented-Lagrangian step $p_k$ from
    \begin{equation*}
        \big(\beta I_n+\rho\grad c(x_k)\grad c(x_k)^\top\big)p_k
        =-\grad f(x_k)-\grad c(x_k)(\lambda_k+\rho c(x_{k})).
    \end{equation*}
    \State Set $x_{k+1}=x_k+p_k$.
    \State Perform the classical ALM multiplier update
    \begin{equation*}
        \lambda_{k+1}=\lambda_k+\rho c(x_{k+1}).
    \end{equation*}
\EndFor
\end{algorithmic}
\end{algorithm}

After the initial evaluation of $c(x_0)$, each iteration evaluates
$\grad f(x_k)$, $\grad c(x_k)$, and the new constraint residual
$c(x_{k+1})$, and solves one Jacobian-induced symmetric positive-definite system.  Thus no nonlinear inner subproblem is required.  When $m\ll n$, the Woodbury identity reduces this solve to an $m\times m$ positive-definite system.  We count exact linear-system solves separately from oracle evaluations; this count does not represent their arithmetic or Krylov cost.

The primal model uses \(c(x_k)+\grad c(x_k)^\top p_k\), whereas the
multiplier update uses \(c(x_{k+1})\).  Define the constraint linearization error by \(d_k:=c(x_{k+1})-c(x_k)-\grad c(x_k)^\top p_k\). Combining the first-order optimality condition of the primal model with the classical multiplier update gives
\begin{equation}
\grad f(x_k) +\grad c(x_k)\lambda_{k+1} +\beta p_k = \rho\grad c(x_k)d_k.
\label{eq:identity}
\end{equation}
Whenever \([x_k,x_{k+1}]\subset\mathcal U\), the fundamental theorem of calculus gives
$$
d_k = \int_0^1 \left[ \grad c(x_k+t p_k)-\grad c(x_k) \right]^\top p_k\,dt.
$$
Hence, by the Lipschitz continuity of \(\grad c\),
\begin{equation}
    \norm{d_k}
    \le
    \kappa_{\mathrm{lin}}\norm{p_k}^2,
    \qquad
    \kappa_{\mathrm{lin}}
    :=
    \frac{L_c}{2}.
    \label{eq:linerrorbounds}
\end{equation}

Thus the classical update introduces a quadratic perturbation into the multiplier identity that would be exact under a linearized residual update. We say that the iterates remain in the regularity region up to iteration $N$ if $[x_k,x_{k+1}]\subset\mathcal U$ for $0\le k\le N-1$. Requiring the entire line segment between successive iterates, rather than only its endpoints, to lie in $\mathcal U$ permits the use of the local Taylor estimates. This condition is automatic when $\mathcal U=\R^n$; under the conditions stated below, Theorem~\ref{thm:localization} shows that it holds for every $N$.

For the fixed parameters $\beta,\rho>0$ and auxiliary constants
$\Delta,\Lambda>0$, define
$$
\begin{aligned}
    C_{\mathrm{mod}}
    &:=\frac{L_f}{2}
      +\kappa_{\mathrm{lin}}(3\Lambda+\rho M\Delta)
      +\frac{\rho}{2}\kappa_{\mathrm{lin}}^2\Delta^2,\quad
    A_p:=\beta+\rho M\kappa_{\mathrm{lin}}\Delta,\\
    B_p&:=A_p+L_f+L_c\Lambda,\quad
    C_{\lambda,p}:=\frac{4}{\sigma^2}\max\{A_p^2,B_p^2\}.
\end{aligned}
$$
Here $C_{\mathrm{mod}}$ collects the error terms in the primal-descent estimate; $A_p$ and $B_p$ are the coefficients of $\norm{p_k}$ and $\norm{p_{k-1}}$, respectively, in the multiplier-increment estimate; and $C_{\lambda,p}$ is the resulting coefficient in the squared multiplier-increment bound.

\begin{assumption}
\label{ass:safe}
\leanDeclTag{Assumption_2_3.lean}{174}
\leanDeclTagEnd
The fixed algorithmic parameters $\beta,\rho>0$ and auxiliary constants $\Delta,\Lambda>0$ are independent of the target accuracy and satisfy
\begin{equation}
\Lambda
\ge\frac{G+\beta\Delta
+\rho M\kappa_{\mathrm{lin}}\Delta^2}{\sigma},
\quad
\frac{G}{\beta}+\frac{3M\Lambda}{\beta+\rho\sigma^2}
\le\Delta,
\quad
C_{\mathrm{mod}}
\le\frac{3\beta}{8},
\quad
\rho
\ge\frac{8C_{\lambda,p}}{\beta}.
\label{eq:safe}
\end{equation}
The initial pair also satisfies $\norm{\lambda_0}\le\Lambda$ and
$\rho\norm{c(x_0)}\le\Lambda$.
\end{assumption}

Together, the inequalities in \eqref{eq:safe} yield the step and multiplier bounds and the Lyapunov descent estimate used below.  They are sufficient analytical conditions rather than a parameter-selection rule.  The next proposition shows that they are jointly feasible with fixed, accuracy-independent parameters.

\begin{proposition}
\label{prop:existence}
\leanDeclTag{Proposition_2_4.lean}{381}
\leanDeclTagEnd
Under Assumption~\ref{ass:smooth}, there exist constants
$\Delta>0$, $\tau_\rho>0$, and $\beta_0>0$ such that, for every fixed $\beta\ge\beta_0$, setting $\rho=\tau_\rho\beta$, there is a
$\Lambda>0$ for which all inequalities in \eqref{eq:safe} hold.
The resulting tuple $(\Delta,\beta,\rho,\Lambda)$ is independent of the target accuracy.
\end{proposition}

\begin{proof}
Choose
\(\tau_\rho>\max\{32/\sigma^2,6M/\sigma^3\}\) and, for
\(\Delta>0\), set
\begin{align*}
a_\Delta:=\tau_\rho M\kappa_{\mathrm{lin}}\Delta,\quad
H_\Delta:=\frac4{\sigma^2}
\max\left\{
(1+a_\Delta)^2,\
\left(1+a_\Delta+
\frac{L_c\Delta(1+a_\Delta)}{\sigma}\right)^2
\right\}.
\end{align*}
Since \(a_\Delta\to0\) and
\(H_\Delta\to4/\sigma^2<\tau_\rho/8\), choose \(\Delta>0\) so that
\begin{equation}
 a_\Delta\le1,\qquad H_\Delta<\tau_\rho/8,\qquad
 \kappa_{\mathrm{lin}}\Delta
 \left(\frac6\sigma+\tau_\rho M\right)
 +\frac{\tau_\rho}{2}\kappa_{\mathrm{lin}}^2\Delta^2
 \le\frac18.
    \label{eq:delta-choice}
\end{equation}
Now set \(\rho=\tau_\rho\beta\) and
\(\Lambda:=(G+\beta\Delta+\rho M\kappa_{\mathrm{lin}}\Delta^2)/\sigma\).
With this definition of \(\Lambda\), verification of \eqref{eq:safe}
reduces, as \(\beta\to\infty\), to
\begin{align*}
 &\frac1\Delta\left[
 \frac G\beta+\frac{3M\Lambda}{\beta+\rho\sigma^2}\right]
 \longrightarrow
 \frac{3M(1+a_\Delta)}{\sigma(1+\tau_\rho\sigma^2)}
 <1,\\
 &\frac{C_{\mathrm{mod}}}{\beta}
 \le o(1)+\kappa_{\mathrm{lin}}\Delta
 \left(\frac6\sigma+\tau_\rho M\right)
 +\frac{\tau_\rho}{2}\kappa_{\mathrm{lin}}^2\Delta^2
 \le o(1)+\frac18,\\
 &\frac{C_{\lambda,p}}{\beta^2}
 \longrightarrow H_\Delta<\frac{\tau_\rho}{8}.
\end{align*}
Thus, for all sufficiently large \(\beta\), all conditions in \eqref{eq:safe} hold.  Increasing a fixed threshold \(\beta_0\) if necessary completes the proof, with every choice independent of \(\eps\).
\end{proof}

The proof gives the following order of choice: fix $\tau_\rho$ above the stated threshold, choose $\Delta>0$ sufficiently small and $\beta>0$ sufficiently large, and then set $\rho=\tau_\rho\beta$ and $\Lambda=(G+\beta\Delta+\rho M\kappa_{\mathrm{lin}}\Delta^2)/\sigma$. These choices depend only on the constants in Assumption~\ref{ass:smooth} and are independent of the target accuracy $\eps$.  The initialization requirement in Assumption~\ref{ass:safe} remains separate.  We next specify a level set on which the preceding conditional estimates ensure that all subsequent iterates remain in $\mathcal U$.

For $S\subseteq\R^n$ and $r\ge0$, write
$\mathcal N_r(S):=\{y\in\R^n:\operatorname{dist}(y,S)\le r\}$.
Define
\begin{equation}
    \overline\Phi_1
    :=
    f(x_0)+G\Delta+\frac{4\Lambda^2}{\rho}
    +\frac{C_{\lambda,p}}{\rho}\Delta^2,
    \qquad
    H_{\mathrm{det}}
    :=
    \overline\Phi_1+\frac{\Lambda^2}{2\rho}.
    \label{eq:localizationlevels}
\end{equation}
Set
\(\mathcal S_{\mathrm{det}}
:=\{x\in\R^n:f(x)\le H_{\mathrm{det}},
\norm{c(x)}\le2\Lambda/\rho\}\).

\begin{assumption}
\label{ass:localization}
\leanDeclTag{Assumption_2_5/Region.lean}{78}
\leanDeclTagEnd
The closed $\Delta$-neighborhood of $\mathcal S_{\mathrm{det}}$ is contained in the regularity region:
$\mathcal N_\Delta(\mathcal S_{\mathrm{det}})\subset\mathcal U$.
\end{assumption}

This condition is automatic when $\mathcal U=\R^n$.  More generally, it suffices that the bounds in Assumption~\ref{ass:smooth} hold on an open set containing the closed $\Delta$-neighborhood of $\mathcal S_{\mathrm{det}}$; no global regularity is required. The initialization in Assumption~\ref{ass:safe} holds, in particular, whenever $x_0$ is feasible and $\lambda_0=0$.  The following lemma derives uniform step and multiplier bounds up to any iteration for which the iterates remain in the regularity region.

\begin{lemma}
\label{lem:invariant}
\leanDeclTag{Lemma_2_6.lean}{373}
\leanDeclTagEnd
Suppose that Assumptions~\ref{ass:smooth} and~\ref{ass:safe} hold, and let $N\ge1$. If the NR-LALM iterates remain in the regularity region up to iteration $N$, then
$$
    \norm{p_k}\leq\Delta
    \quad (0\leq k\leq N-1),
    \qquad
    \norm{\lambda_k}\leq\Lambda
    \quad (0\leq k\leq N).
$$
\end{lemma}
\begin{proof}
For any $A\in\R^{n\times m}$,
$$
    (\beta I_n+\rho AA^\top)^{-1}A
    =
    A(\beta I_m+\rho A^\top A)^{-1}.
$$
The primal-step equation and uniform LICQ therefore give
\begin{equation}
    \norm{p_k}
    \leq
    \frac{G}{\beta}
    +
    \frac{M}{\beta+\rho\sigma^2}
    \norm{\lambda_k+\rho c(x_k)}.
    \label{eq:p-resolvent}
\end{equation}
At $k=0$, Assumption~\ref{ass:safe} gives $\norm{\lambda_0+\rho c(x_0)}\leq2\Lambda$.
Hence \eqref{eq:p-resolvent} and \eqref{eq:safe} imply
$\norm{p_0}\leq\Delta$.  Moreover, by
\eqref{eq:identity}, \eqref{eq:linerrorbounds}, and uniform LICQ,
$$
    \sigma\norm{\lambda_1}
    \leq
    G+\beta\Delta
    +\rho M\kappa_{\mathrm{lin}}\Delta^2
    \leq
    \sigma\Lambda.
$$
Now suppose that
$\norm{\lambda_{k-1}},\norm{\lambda_k}\leq\Lambda$. The multiplier
update yields
$$
    \rho c(x_k)=\lambda_k-\lambda_{k-1},
    \qquad
    \norm{\lambda_k+\rho c(x_k)}
    =
    \norm{2\lambda_k-\lambda_{k-1}}
    \leq3\Lambda.
$$
Thus \eqref{eq:p-resolvent} and \eqref{eq:safe} give
$\norm{p_k}\leq\Delta$. Applying
\eqref{eq:identity}, \eqref{eq:linerrorbounds}, and uniform LICQ once
more gives
$$
    \sigma\norm{\lambda_{k+1}}
    \leq
    G+\beta\Delta
    +\rho M\kappa_{\mathrm{lin}}\Delta^2
    \leq
    \sigma\Lambda.
$$
The conclusion follows by induction.
\end{proof}

Here $\Lambda$ is an a priori analysis constant, but Lemma~\ref{lem:invariant} proves---rather than assumes---that every multiplier up to any iteration for which the iterates remain in the regularity region satisfies this bound.

\subsection{Lyapunov descent and trajectory localization}

The primal step decreases the augmented Lagrangian before the multiplier update, whereas the classical multiplier update adds a nonnegative term. The next two lemmas quantify these two effects.  Their combination yields a decreasing Lyapunov function \eqref{eq:potential}, which is then used to establish trajectory localization.

\begin{lemma}
\label{lem:primaldescent}
\leanDeclTag{Lemma_2_7.lean}{316}
\leanDeclTagEnd
Under Assumptions~\ref{ass:smooth} and~\ref{ass:safe}, for any $k\ge0$ such that the iterates remain in the regularity region up to iteration $k+1$, we have
\(\cL_\rho(x_{k+1},\lambda_k)
\le\cL_\rho(x_k,\lambda_k)-(\beta/2)\norm{p_k}^2\).
\end{lemma}

\begin{proof}
Using the $L_f$-smoothness of $f$,
$$
    c(x_{k+1})
    =
    c(x_k)+\grad c(x_k)^\top p_k+d_k,
$$
and the inner product of the model optimality condition with $p_k$,
we obtain
\begin{align*}
\cL_\rho(x_{k+1},\lambda_k)
-\cL_\rho(x_k,\lambda_k)
\leq
&-\left(\beta-\frac{L_f}{2}\right)\norm{p_k}^2
-\frac{\rho}{2}\norm{\grad c(x_k)^\top p_k}^2                 \\
&+\ip{\lambda_k+\rho c(x_k)
+\rho\grad c(x_k)^\top p_k}{d_k}
+\frac{\rho}{2}\norm{d_k}^2.
\end{align*}
By Lemma~\ref{lem:invariant} and \eqref{eq:linerrorbounds},
\begin{align*}
&\left|
\ip{\lambda_k+\rho c(x_k)
+\rho\grad c(x_k)^\top p_k}{d_k}
\right|
+\frac{\rho}{2}\norm{d_k}^2                                  \\
&\leq
\left[
\kappa_{\mathrm{lin}}(3\Lambda+\rho M\Delta)
+\frac{\rho}{2}\kappa_{\mathrm{lin}}^2\Delta^2
\right]\norm{p_k}^2.
\end{align*}
Therefore, by the definition of $C_{\mathrm{mod}}$,
$$
    \cL_\rho(x_{k+1},\lambda_k)
    -\cL_\rho(x_k,\lambda_k)
    \leq
    (-\beta+C_{\mathrm{mod}})\norm{p_k}^2
    \leq
    -\frac{5\beta}{8}\norm{p_k}^2,
$$
where the last inequality follows from \eqref{eq:safe}.  The result
follows since $5\beta/8\geq\beta/2$.
\end{proof}

The dual update changes the augmented Lagrangian by
\begin{equation}
    \cL_\rho(x_{k+1},\lambda_{k+1})
    -\cL_\rho(x_{k+1},\lambda_k)
    =\frac1\rho\norm{\lambda_{k+1}-\lambda_k}^2.
    \label{eq:dualincrease}
\end{equation}
Thus the multiplier update contributes a nonnegative term.  To absorb this term into the primal decrease, we next bound the multiplier increment by the two most recent primal steps.
\begin{lemma}
\label{lem:dualinc}
\leanDeclTag{Lemma_2_8.lean}{286}
\leanDeclTagEnd
Under Assumptions~\ref{ass:smooth} and~\ref{ass:safe}, for any $k\ge1$ such that the iterates remain in the regularity region up to iteration $k+1$, the following bound holds:
\begin{equation}
    \norm{\lambda_{k+1}-\lambda_k}^2
    \le C_{\lambda,p}
       \big(\norm{p_k}^2+\norm{p_{k-1}}^2\big).
    \label{eq:dualinc}
\end{equation}
\end{lemma}

\begin{proof}
Subtracting \eqref{eq:identity} at \(k-1\) from that at \(k\) gives
\begin{align*}
\grad c(x_k)(\lambda_{k+1}-\lambda_k)
={}&-\big(\grad f(x_k)-\grad f(x_{k-1})\big)
    -\beta(p_k-p_{k-1})\\
&-\big(\grad c(x_k)-\grad c(x_{k-1})\big)\lambda_k
+\rho\grad c(x_k)d_k
  -\rho\grad c(x_{k-1})d_{k-1}.
\end{align*}
The objective-gradient difference is bounded by
\(L_f\norm{p_{k-1}}\), the Jacobian-difference term by
\(L_c\Lambda\norm{p_{k-1}}\), and each remainder term by
\(\rho M\kappa_{\mathrm{lin}}\Delta\norm{p_j}\) for \(j=k-1,k\).
Uniform LICQ therefore gives
\(\sigma\norm{\lambda_{k+1}-\lambda_k}
\le A_p\norm{p_k}+B_p\norm{p_{k-1}}\).
Squaring and using the definition of \(C_{\lambda,p}\) proves
\eqref{eq:dualinc}.
\end{proof}

The estimate \eqref{eq:dualinc} contains the lagged step
$\norm{p_{k-1}}^2$.  To absorb this term when combining the primal decrease with \eqref{eq:dualincrease}, we add a one-step memory term to the augmented Lagrangian and define, for $k\geq1$,
\begin{equation}
    \Phi_k
    :=\cL_\rho(x_k,\lambda_k)
      +\frac{C_{\lambda,p}}{\rho}\norm{p_{k-1}}^2.
    \label{eq:potential}
\end{equation}

\begin{theorem}
\label{thm:lyap}
\leanDeclTag{Theorem_2_9.lean}{97}
\leanDeclTagExtra{Theorem_2_9.lean}{136}
\leanDeclTagEnd
Suppose Assumptions~\ref{ass:smooth} and~\ref{ass:safe} hold, and let $N\ge2$. If the iterates remain in the regularity region up to iteration $N$, then, for $1\le k\le N-1$,
\begin{equation}
    \Phi_{k+1}
    \le \Phi_k
      -\frac{\beta}{4}\norm{p_k}^2.
    \label{eq:lyap}
\end{equation}
Moreover,
\begin{equation}
    \Phi_k\ge
    \underline\Phi
    :=\underline{f}-\frac{\Lambda^2}{2\rho}
    \qquad \text{for }1\le k\le N.
    \label{eq:lowerphi}
\end{equation}
\end{theorem}

\begin{proof}
Combine Lemma~\ref{lem:primaldescent}, \eqref{eq:dualincrease}, and
Lemma~\ref{lem:dualinc}, and then add the potential term in
\eqref{eq:potential}.  This gives
\begin{align*}
    \Phi_{k+1}
    \le{}&\Phi_k
    +\left(-\frac\beta2+\frac{2C_{\lambda,p}}\rho\right)\norm{p_k}^2.
\end{align*}
The conditions in \eqref{eq:safe} imply
$2C_{\lambda,p}/\rho\le\beta/4$, proving \eqref{eq:lyap}.

For the lower bound, completing the square gives
\(\ip{\lambda_k}{c(x_k)}+(\rho/2)\norm{c(x_k)}^2
\ge-\norm{\lambda_k}^2/(2\rho)\ge-\Lambda^2/(2\rho)\).
The second term in \eqref{eq:potential} is nonnegative, so \eqref{eq:lowerphi} follows.
\end{proof}

The Lyapunov inequality above is conditional on the iterates remaining in the regularity region. The following result uses Assumption~\ref{ass:localization} to verify this condition for every iteration.

\begin{theorem}
\label{thm:localization}
\leanDeclTag{Theorem_2_10.lean}{276}
\leanDeclTagExtra{Theorem_2_10.lean}{767}
\leanDeclTagEnd
Let Assumptions~\ref{ass:smooth}, \ref{ass:safe}, and~\ref{ass:localization}
hold, and run NR-LALM in Algorithm~\ref{alg:lalm}.  Then NR-LALM is
well defined for every $k$, $[x_k,x_{k+1}]\subset\mathcal U$ for every $k\ge0$,
and
\begin{equation}
    \norm{p_k}\le\Delta,\quad
    \norm{\lambda_k}\le\Lambda,\quad
    x_k\in\mathcal S_{\mathrm{det}}\ (k\ge0),\quad
    \Phi_k\le\overline\Phi_1\ (k\ge1).
    \label{eq:localizedbounds}
\end{equation}
\end{theorem}

\begin{proof}
By Assumption~\ref{ass:localization},
\begin{equation}
    x_k\in\mathcal S_{\mathrm{det}},
    \quad
    \norm{p_k}\leq\Delta
    \quad\Longrightarrow\quad
    [x_k,x_{k+1}]
    \subset
    \mathcal N_\Delta(\mathcal S_{\mathrm{det}})
    \subset\mathcal U.
    \label{eq:one-step-localization}
\end{equation}
The initialization conditions and \eqref{eq:localizationlevels} give
$x_0\in\mathcal S_{\mathrm{det}}$.  Moreover,
\eqref{eq:p-resolvent} and \eqref{eq:safe} give
$\norm{p_0}\leq\Delta$. Hence \eqref{eq:one-step-localization}
gives $[x_0,x_1]\subset\mathcal U$, and
\eqref{eq:identity}--\eqref{eq:linerrorbounds} give
$\norm{\lambda_1}\leq\Lambda$.  Moreover,
$$
    \rho\norm{c(x_1)}
    =
    \norm{\lambda_1-\lambda_0}
    \leq2\Lambda,
$$
and therefore
\begin{align*}
\Phi_1
&\leq
f(x_0)+G\Delta
+\frac{4\Lambda^2}{\rho}
+\frac{C_{\lambda,p}}{\rho}\Delta^2
=
\overline\Phi_1.
\end{align*}
Completing the square in the augmented Lagrangian gives
$f(x_1) \leq \Phi_1+\Lambda^2/(2\rho) \leq H_{\mathrm{det}}$;
hence $x_1\in\mathcal S_{\mathrm{det}}$.

Suppose now that the claims hold through some $k\geq1$.  Since
$\rho c(x_k)=\lambda_k-\lambda_{k-1}$, we have
$\norm{\lambda_k+\rho c(x_k)} \leq3\Lambda$.
Thus \eqref{eq:p-resolvent} and \eqref{eq:safe} give
$\norm{p_k}\leq\Delta$. The implication
\eqref{eq:one-step-localization} gives $[x_k,x_{k+1}]\subset\mathcal U$,
and \eqref{eq:identity}--\eqref{eq:linerrorbounds} give
$\norm{\lambda_{k+1}}\leq\Lambda$.  Therefore
$\rho\norm{c(x_{k+1})}
=\norm{\lambda_{k+1}-\lambda_k}\le2\Lambda$.  Finally,
Theorem~\ref{thm:lyap} yields
$$
    \Phi_{k+1}\leq\Phi_k\leq\overline\Phi_1,
$$
and completing the square once more gives
$$
    f(x_{k+1})
    \leq
    \Phi_{k+1}+\frac{\Lambda^2}{2\rho}
    \leq
    H_{\mathrm{det}}.
$$
Thus $x_{k+1}\in\mathcal S_{\mathrm{det}}$.
The conclusion follows by induction.
\end{proof}

\subsection{Complexity and full-sequence convergence}

The KKT residual can be controlled by two consecutive primal steps.
Combining this estimate with the Lyapunov descent inequality yields the deterministic complexity result.  We then establish finite length and full-sequence convergence under compactness and the KL property.  For the generated primal--dual pairs, write
$\mathcal R_k:=\mathcal R(x_k,\lambda_k)$.

\begin{lemma}
\label{lem:kktres}
\leanDeclTag{Lemma_2_11.lean}{237}
\leanDeclTagEnd
Under Assumptions~\ref{ass:smooth} and~\ref{ass:safe}, let $N\ge2$ and suppose that the deterministic NR-LALM iterates remain in the regularity region up to iteration $N$. Then
\(\mathcal R_{k+1}^2
\le C_R(\norm{p_k}^2+\norm{p_{k-1}}^2)\) for \(1\le k\le N-1\),
where \(C_s:=\beta+\rho M\kappa_{\mathrm{lin}}\Delta+L_f+L_c\Lambda\)
and \(C_R:=C_s^2+C_{\lambda,p}/\rho^2\).
\end{lemma}

\begin{proof}
To bound the stationarity residual at $(x_{k+1},\lambda_{k+1})$, we combine \eqref{eq:identity} with the objective-gradient and Jacobian differences:
\begin{align*}
&\grad f(x_{k+1})+\grad c(x_{k+1})\lambda_{k+1}\\
&=-\beta p_k+\rho\grad c(x_k)d_k
  +\grad f(x_{k+1})-\grad f(x_k)
  +\big(\grad c(x_{k+1})-\grad c(x_k)\big)\lambda_{k+1}.
\end{align*}
Using \eqref{eq:linerrorbounds} and Lemma~\ref{lem:invariant},
\begin{equation}
\norm{\grad f(x_{k+1})
      +\grad c(x_{k+1})\lambda_{k+1}}
\le C_s\norm{p_k}.
\label{eq:statbound}
\end{equation}
The multiplier update and Lemma~\ref{lem:dualinc} give
\(\norm{c(x_{k+1})}^2
\le(C_{\lambda,p}/\rho^2)(\norm{p_k}^2+\norm{p_{k-1}}^2)\).
Combining the two inequalities proves the claim.
\end{proof}

\begin{theorem}
\label{thm:det}
\leanDeclTag{Theorem_2_12.lean}{337}
\leanDeclTagExtra{Theorem_2_12.lean}{423}
\leanDeclTagEnd
Let Assumptions~\ref{ass:smooth}, \ref{ass:safe}, and~\ref{ass:localization}
hold, and run NR-LALM in Algorithm~\ref{alg:lalm} with exact gradients.  For
\(K\ge2\), let
\(\widehat{k}\) be uniformly distributed over \(\{1,\ldots,K-1\}\).  Then
\begin{equation}
    \min_{1\le k\le K-1}\mathcal R_{k+1}^2
    \le \E[\mathcal{R}_{\widehat{k}+1}^2]
    \le \frac{C_{\mathrm{det}}}{K-1},
    \label{eq:detrate}
\end{equation}
where
\(C_{\mathrm{det}}
:=C_R[\Delta^2+8(\Phi_1-\underline\Phi)/\beta]\)
is independent of \(K\) and \(\eps\).  In particular, if
\(K-1\ge C_{\mathrm{det}}\eps^{-2}\), choose \(k_\star\) to minimize
\(\mathcal R_{k+1}\) over \(1\le k\le K-1\).  Then
\((x_{k_\star+1},\lambda_{k_\star+1})\) is an \(\eps\)-KKT pair.  Thus
an \(\eps\)-KKT pair is obtained using \(\cO(\eps^{-2})\) iterations,
\(\cO(\eps^{-2})\) first-order oracle evaluations, and
\(\cO(\eps^{-2})\) exact Jacobian-induced linear-system solves, with a fixed
penalty parameter and the classical ALM multiplier update.
\end{theorem}

\begin{proof}
By Theorem~\ref{thm:lyap},
\((\beta/4)\sum_{k=1}^{K-1}\norm{p_k}^2
\le \Phi_1-\underline\Phi\).
Since \(\norm{p_0}\le\Delta\), Lemma~\ref{lem:kktres} gives
\begin{align*}
    \sum_{k=1}^{K-1}\mathcal{R}_{k+1}^2
    &\le C_R\left(
       \sum_{k=1}^{K-1}\norm{p_k}^2
       +\sum_{k=1}^{K-1}\norm{p_{k-1}}^2\right)\\
    &\le C_R\left(\Delta^2+
       2\sum_{k=1}^{K-1}\norm{p_k}^2\right)
    \le C_{\mathrm{det}}.
\end{align*}
Dividing by \(K-1\) proves the randomized bound.  The minimum is no
larger than the average, which proves the first inequality in
\eqref{eq:detrate} and the asserted pointwise \(\eps\)-KKT guarantee.
\end{proof}

Theorem~\ref{thm:det} establishes the deterministic complexity guarantees without compactness or a KL assumption.  To obtain convergence of the entire primal--dual sequence, we now impose these additional assumptions and establish that the trajectory has finite length.  The KL analysis must account for the two-step multiplier bound in Lemma~\ref{lem:dualinc}, which couples $p_k$ and
$p_{k-1}$.  To accommodate this dependence, we retain the preceding primal step in a lifted state and define the following Lyapunov function:
\begin{equation}
    \mathcal E(x,\lambda,w)
    :=
    \cL_\rho(x,\lambda)
    +\frac{\beta}{4}\norm{w}^2,
    \qquad
    u_k:=(x_k,\lambda_k,p_{k-1}),
    \quad k\geq1,
    \label{eq:KLenergy}
\end{equation}
Write $\mathcal E_k:=\mathcal E(u_k)$.  The coefficient $\beta/4$ is chosen so that the resulting descent estimate controls both
$\norm{p_k}^2$ and $\norm{p_{k-1}}^2$.

Before stating the convergence result, we recall the differentiable KL property.  Let $\mathcal D\subseteq\R^q$ be open.  A continuously
differentiable function $h:\mathcal D\to\R$ has the KL property at
$\bar z\in\mathcal D$ if there exist a neighborhood $\mathcal V\subseteq\mathcal D$ of $\bar z$, $\eta>0$, and a continuous
concave function $\varphi:[0,\eta)\to[0,\infty)$ satisfying
$\varphi(0)=0$, $\varphi\in C^1((0,\eta))$, and $\varphi'(s)>0$ for $s\in(0,\eta)$ such that
$$
    \varphi'\bigl(h(z)-h(\bar z)\bigr)\norm{\grad h(z)}\geq1
$$
whenever $z\in\mathcal V$ and $0<h(z)-h(\bar z)<\eta$.  The function
$\varphi$ is called a desingularizing function.  If $h$ is constant on a compact set and has the KL property at each of its points, the standard uniformization lemma supplies a single desingularizing function and constants that are valid in a neighborhood of the entire set~\cite{AttouchBolteSvaiter2013}.  The proof below applies this lemma to the cluster set of $\{u_k\}$.

\Needspace{9\baselineskip}

\begin{theorem}
\label{thm:KL}
\leanDeclTag{Theorem_2_13.lean}{804}
\leanDeclTagExtra{Theorem_2_13.lean}{1220}
\leanDeclTagEnd
Let Assumptions~\ref{ass:smooth}, \ref{ass:safe}, and~\ref{ass:localization}
hold, and run NR-LALM with
exact gradients.  Assume also that
\(\mathcal S_{\mathrm{det}}\) is compact
and that the function \(\mathcal E\) defined in \eqref{eq:KLenergy} has the
KL property at every point of the cluster set
\(\Omega:=\{\bar u:\text{there is a subsequence }u_{k_j}\to\bar u\}\).
Then the primal--dual trajectory has finite length:
\begin{equation}
 \sum_{k=0}^{\infty}
 \left(\norm{x_{k+1}-x_k}
       +\norm{\lambda_{k+1}-\lambda_k}\right)<\infty,
 \label{eq:KLoriginallength}
\end{equation}
and the entire sequence $(x_k,\lambda_k)$ converges to a KKT pair
$(x_\star,\lambda_\star)$ of \eqref{eq:problem}.  Moreover,
$u_k\to(x_\star,\lambda_\star,0)$.
\end{theorem}

\begin{proof}
Set
\(\delta_{\mathrm{KL}}=\beta/4-C_{\lambda,p}/\rho\ge\beta/8\).
Combining Lemmas~\ref{lem:primaldescent} and~\ref{lem:dualinc} with
\eqref{eq:dualincrease} and adding the memory term in \eqref{eq:KLenergy}
gives
\begin{equation}
 \mathcal E_k-\mathcal E_{k+1}
 \ge\delta_{\mathrm{KL}}
 \left(\norm{p_k}^2+\norm{p_{k-1}}^2\right).
 \label{eq:KLdescent}
\end{equation}
The localization bounds and completion of the square imply that
\(\mathcal E_k\ge \underline{f}-\Lambda^2/(2\rho)\).  Hence
\(\mathcal E_k\downarrow\mathcal E_\star\) and
\begin{equation}
 \sum_{k\ge1}\norm{p_k}^2<\infty,\qquad p_k\to0.
 \label{eq:KLsquaresum}
\end{equation}
Let \(\ell_k=\norm{p_{k-1}}+\norm{\lambda_k-\lambda_{k-1}}\).  Equations \eqref{eq:dualinc} and \eqref{eq:KLdescent} yield
\begin{align}
 \ell_{k+1}^2
 &\le2(1+C_{\lambda,p})
       \left(\norm{p_k}^2+\norm{p_{k-1}}^2\right),
 \label{eq:KLlengthcontrol}\\
 \mathcal E_k-\mathcal E_{k+1}
 &\ge a_{\mathrm{KL}}\ell_{k+1}^2,\qquad
 a_{\mathrm{KL}}:=
 \frac{\delta_{\mathrm{KL}}}{2(1+C_{\lambda,p})}>0.
 \label{eq:KLlengthdescent}
\end{align}
Moreover, \eqref{eq:statbound} with its index shifted by one and
\(\rho c(x_k)=\lambda_k-\lambda_{k-1}\) give, for \(k\ge2\),
\begin{equation}
 \norm{\grad\mathcal E(u_k)}
 \le\left(C_s+\frac\beta2\right)\norm{p_{k-1}}
    +\left(M+\frac1\rho\right)
      \norm{\lambda_k-\lambda_{k-1}} 
 \le b_{\mathrm{KL}}\ell_k,
 \label{eq:KLrelativeerror}
\end{equation}
where \(b_{\mathrm{KL}}=C_s+\beta/2+M+1/\rho\).

By Theorem~\ref{thm:localization} and the assumed compactness, every \(u_k\)
lies in the compact product set
\(\mathcal S_{\mathrm{det}}
\times\{\lambda:\norm{\lambda}\le\Lambda\}
\times\{p:\norm{p}\le\Delta\}\).  Thus its cluster set \(\Omega\) is
nonempty and compact, \(\operatorname{dist}(u_k,\Omega)\to0\), and continuity
of \(\mathcal E\), together with \(\mathcal E_k\downarrow\mathcal E_\star\),
gives \(\mathcal E=\mathcal E_\star\) on \(\Omega\).  Equations
\eqref{eq:dualinc}, \eqref{eq:KLsquaresum}, and
\eqref{eq:KLrelativeerror} also give
\(\lambda_k-\lambda_{k-1}\to0\) and
\(\grad\mathcal E(u_k)\to0\).
If \(\mathcal E_k=\mathcal E_\star\) for some \(k\), then
\eqref{eq:KLdescent}, \eqref{eq:dualinc}, and \eqref{eq:identity} show
that the tail is constant at a KKT pair.

Otherwise, set
\(\zeta_k=\mathcal E_k-\mathcal E_\star>0\).  Uniformization of the KL
property on \(\Omega\) gives \(k_0\) and \(\varphi\) such that, for
\(k\ge k_0\), concavity, \eqref{eq:KLlengthdescent}, and
\eqref{eq:KLrelativeerror} imply
\(\varphi(\zeta_k)-\varphi(\zeta_{k+1})
\ge a_{\mathrm{KL}}\ell_{k+1}^2/(b_{\mathrm{KL}}\ell_k)\).
Here \(\ell_k>0\), since otherwise the KL inequality would contradict
\(\zeta_k>0\).  The arithmetic--geometric mean inequality gives
\begin{equation}
 \ell_{k+1}\le\frac12\ell_k+
 \frac{b_{\mathrm{KL}}}{2a_{\mathrm{KL}}}
 [\varphi(\zeta_k)-\varphi(\zeta_{k+1})].
 \label{eq:KLrecursion}
\end{equation}
Summation and telescoping yield
\(\sum_{k=k_0}^{N}\ell_{k+1}\le
\ell_{k_0}+(b_{\mathrm{KL}}/a_{\mathrm{KL}})
\varphi(\zeta_{k_0})\), uniformly in \(N\).  Thus
\eqref{eq:KLoriginallength} holds because \(x_{k+1}-x_k=p_k\).  Hence
both primal and multiplier sequences are Cauchy.  Their limit satisfies
feasibility by the multiplier update and stationarity by
\eqref{eq:statbound}; hence it is a KKT pair, and
\(p_{k-1}\to0\) gives the asserted convergence of \(u_k\).
\end{proof}

The KL assumption is satisfied, in particular, when $f$ and $c$ are
semialgebraic.  Indeed, $\mathcal E$ is then semialgebraic and hence has the KL property ~\cite{BolteDaniilidisLewis2007,AttouchBolteSvaiter2013}.

\section{Stochastic linearized augmented Lagrangian method}
\label{sec:stochastic}

We now replace the exact objective gradient in NR-LALM by a projected SPIDER estimator.  The direct stochastic analysis requires $[x_k,x_{k+1}]\subset\mathcal U$ at each iteration; this condition is automatic when $\mathcal U=\R^n$.  Under local regularity, we control the finite-horizon exit probability through a stopped-process argument and obtain an unconditional expected guarantee by safeguarded restarts based on an exact membership test for a prescribed localization set.  In addition to
Assumptions~\ref{ass:smooth} and~\ref{ass:safe}, we impose standard
unbiasedness and bounded-variance conditions~\cite{GhadimiLan2013}, together with the mean-squared smoothness condition required by the SPIDER recursion.

\begin{assumption}
\label{ass:stoch}
\leanDeclTag{Assumption_3_1/Oracle.lean}{19}
\leanDeclTagEnd
Suppose that \(f(x)=\E[F(x;\xi)]\) for \(x\in\mathcal U\).  All random
variables introduced below are assumed to be measurable.
For almost every \(\xi\), \(F(\cdot;\xi)\) is differentiable on \(\mathcal U\).
There exist constants \(\sigma_f,L_s\ge0\) such that:
\begin{enumerate}[label=\textup{(\roman*)},leftmargin=2.5em,itemsep=0.15em,topsep=0.25em]
\item \(\E[\grad F(x;\xi)]=\grad f(x)\) and
      \(\E\norm{\grad F(x;\xi)-\grad f(x)}^2\le\sigma_f^2\) for all
      \(x\in\mathcal U\);
\item \(\E\norm{\grad F(x;\xi)-\grad F(y;\xi)}^2
      \le L_s^2\norm{x-y}^2\) for all \(x,y\in\mathcal U\).
\end{enumerate}
\end{assumption}

Condition~\textup{(ii)} is the mean-squared smoothness condition used in variance-reduced stochastic methods~\cite{ShiWangWang2025}. The same sample \(\xi\) is used at \(x\) and \(y\), as required by the SPIDER difference estimator.  By condition~\textup{(i)}, Jensen's inequality, and condition~\textup{(ii)}, \(\grad f\) is \(L_s\)-Lipschitz continuous on \(\mathcal U\), without requiring a uniform sample-wise Lipschitz constant. For a random output, we measure approximate stationarity in mean square, following the convention used in stochastic constrained optimization \cite{ShiWangWang2025,LiEtAl2024Stoc}.

\begin{definition}
\label{def:stoch-eps-kkt}
\leanDeclTag{Definition_3_2/Stochastic.lean}{66}
\leanDeclTagEnd
Given \(\eps\ge0\), a random point \(x\) is a stochastic \(\eps\)-KKT point if there exists a random multiplier \(\lambda\), defined on the same probability space, such that \(\E\norm{\grad f(x)+\grad c(x)\lambda}^2\le\eps^2\) and \(\E\norm{c(x)}^2\le\eps^2\).  A specified random pair \((x,\lambda)\) satisfying these inequalities is called a stochastic \(\eps\)-KKT pair.
\end{definition}

Since \(\mathcal R(x,\lambda)^2\) is the sum of the two squared residuals, the single bound \(\E[\mathcal R(x,\lambda)^2]\le\eps^2\) implies the conditions in Definition~\ref{def:stoch-eps-kkt}.  Conversely, every stochastic \(\eps\)-KKT pair satisfies \(\E[\mathcal R(x,\lambda)^2]\le2\eps^2\).  Moreover, Jensen's inequality
gives \(\E[\mathcal R(x,\lambda)]\le\eps\) under this single residual bound~\cite{AlacaogluWright2024,CuiEtAl2026}.

Fix a positive integer horizon \(K\) and positive integers \(Q,B,b\).  Let \(\{\mathcal G_k\}_{k\ge0}\) be the filtration generated by the algorithmic history before the \(k\)th batch is sampled.  Then \(x_k\) and, at noncheckpoint iterations, \(h_{k-1}\) are \(\mathcal G_k\)-measurable.  At iteration \(k\), let \(\{\xi_{k,i}\}\) be a batch of independent copies of \(\xi\), independent of \(\mathcal G_k\), with size \(B\) at checkpoints and
\(b\) otherwise.  The raw SPIDER estimator~\cite{FangEtAl2018} is
\begin{equation}
h_k=
\begin{cases}
\displaystyle \frac1B\sum_{i=1}^{B}\grad F(x_k;\xi_{k,i}),
    & k\equiv0\pmod Q,\\[2mm]
\displaystyle h_{k-1}+\frac1b\sum_{i=1}^{b}
\big[\grad F(x_k;\xi_{k,i})-\grad F(x_{k-1};\xi_{k,i})\big],
    & \text{otherwise}.
\end{cases}
\label{eq:spider}
\end{equation}
At noncheckpoint iterations, the same sample \(\xi_{k,i}\) is used in the two gradient evaluations forming each difference.

The raw estimator need not be pathwise bounded, so we project it before the primal step.  We assume that the bound \(G\) in Assumption~\ref{ass:smooth} is available to the stochastic algorithm; any larger valid bound may be used,
provided that the same value is used in the conditions of Assumption~\ref{ass:safe}. Set \(v_k:=\Proj_{\{w\in\R^n:\norm{w}\le G\}}(h_k)\), so that \(\norm{v_k}\le G\) pathwise, and define the raw and projected errors by \(\widehat e_k:=h_k-\grad f(x_k)\) and
\(e_k:=v_k-\grad f(x_k)\), respectively.  Whenever \(x_k\in\mathcal U\), Assumption~\ref{ass:smooth} gives \(\norm{\grad f(x_k)}\le G\); hence nonexpansiveness of the Euclidean projection gives
\begin{equation}
    \norm{e_k}\le \norm{\widehat e_k}.
    \label{eq:projectionerror}
\end{equation}

The stochastic NR-LALM replaces \(\grad f(x_k)\) in Algorithm~\ref{alg:lalm} by \(v_k\); all other primal and multiplier updates remain unchanged.  Its model optimality condition gives
\begin{equation}
    v_k+\grad c(x_k)\lambda_{k+1}+\beta p_k
    =\rho\grad c(x_k)d_k.
    \label{eq:stoch-identity}
\end{equation}
For the stochastic analysis up to horizon $K$, assume that, outside a single null event, $[x_k,x_{k+1}]\subset\mathcal U$ for $0\le k\le K-1$. Since \(\norm{v_k}\le G\),
the proof of Lemma~\ref{lem:invariant} applies to each such sample path with
\(\grad f(x_k)\) replaced by \(v_k\).  Consequently, $\norm{p_k}\le\Delta$ and $\norm{\lambda_k}\le\Lambda$ almost surely.  In particular, the step square-integrability required in
Lemma~\ref{lem:spider} is automatic.

\begin{lemma}
\label{lem:spider}
\leanDeclTag{Lemma_3_3.lean}{1532}
\leanDeclTagEnd
Under Assumptions~\ref{ass:smooth} and~\ref{ass:stoch}, suppose that, almost surely, the stochastic iterates remain in the regularity region up to iteration $K$ and that
\(\E\norm{p_k}^2<\infty\) for \(0\le k<K\). Then
\begin{equation}
    \sum_{k=0}^{K-1}\E\norm{e_k}^2
    \le \frac{K\sigma_f^2}{B}
       +\frac{QL_s^2}{b}
        \sum_{k=0}^{K-1}\E\norm{p_k}^2.
    \label{eq:spidersum}
\end{equation}
\end{lemma}

\begin{proof}
At a checkpoint, conditional independence and
Assumption~\ref{ass:stoch}\textup{(i)} give
$$
\E[\norm{\widehat e_k}^2\mid\mathcal G_k]
=\frac{1}{B^2}\sum_{i=1}^B
  \E[\norm{\grad F(x_k;\xi_{k,i})-\grad f(x_k)}^2\mid\mathcal G_k]
\le\frac{\sigma_f^2}{B}.
$$
At a noncheckpoint iteration, set
$$
\zeta_{k,i}:=\grad F(x_k;\xi_{k,i})-\grad F(x_{k-1};\xi_{k,i})
-\big(\grad f(x_k)-\grad f(x_{k-1})\big).
$$
Then \(\E[\zeta_{k,i}\mid\mathcal G_k]=0\) and
\(\widehat e_k=\widehat e_{k-1}+b^{-1}\sum_{i=1}^b\zeta_{k,i}\).
The conditional cross terms therefore vanish, while
Assumption~\ref{ass:stoch}\textup{(ii)} and the conditional variance identity yield
\begin{align*}
\E[\norm{\widehat e_k}^2\mid\mathcal G_k]
&=\norm{\widehat e_{k-1}}^2
  +\frac{1}{b^2}\sum_{i=1}^b
    \E[\norm{\zeta_{k,i}}^2\mid\mathcal G_k]\\
&\le\norm{\widehat e_{k-1}}^2
  +\frac{L_s^2}{b}\norm{x_k-x_{k-1}}^2
=\norm{\widehat e_{k-1}}^2
  +\frac{L_s^2}{b}\norm{p_{k-1}}^2.
\end{align*}
Taking expectations and writing
\(s(k):=Q\lfloor k/Q\rfloor\) for the most recent checkpoint, iterating
within the current epoch gives
$$
\E\norm{\widehat e_k}^2
\le \frac{\sigma_f^2}{B}
  +\frac{L_s^2}{b}\sum_{j=s(k)}^{k-1}\E\norm{p_j}^2.
$$
After summing over \(k=0,\ldots,K-1\), each \(\E\norm{p_j}^2\) in the
double sum is counted at most \(Q\) times.  Hence
$$
\sum_{k=0}^{K-1}\E\norm{\widehat e_k}^2
\le \frac{K\sigma_f^2}{B}
  +\frac{QL_s^2}{b}\sum_{k=0}^{K-1}\E\norm{p_k}^2.
$$
Finally, \eqref{eq:projectionerror} transfers this bound from the raw errors
\(\widehat e_k\) to the projected errors \(e_k\).
\end{proof}

The estimator error contributes \(-\ip{e_k}{p_k}\) to the deterministic
descent calculation.  Since \(v_k=\grad f(x_k)+e_k\), Young's inequality
bounds this term by \((\beta/8)\norm{p_k}^2\) plus
\((2/\beta)\norm{e_k}^2\), and hence
$$
    \cL_\rho(x_{k+1},\lambda_k)
    \le\cL_\rho(x_k,\lambda_k)
       -\frac{\beta}{2}\norm{p_k}^2
       +\frac{2}{\beta}\norm{e_k}^2.
$$

For \(k\ge1\), subtracting \eqref{eq:stoch-identity} at \(k-1\) from that
at \(k\), and using
\(v_k-v_{k-1}
=\grad f(x_k)-\grad f(x_{k-1})+e_k-e_{k-1}\), gives
$$
\sigma\norm{\lambda_{k+1}-\lambda_k}
\le A_p\norm{p_k}+B_p\norm{p_{k-1}}
   +\norm{e_k}+\norm{e_{k-1}}.
$$
Set $C_{\lambda,e}:=4/\sigma^2$.  Squaring the preceding estimate yields
\begin{equation}
\norm{\lambda_{k+1}-\lambda_k}^2
\le C_{\lambda,p}\bigl(\norm{p_k}^2+\norm{p_{k-1}}^2\bigr)
+C_{\lambda,e}\bigl(\norm{e_k}^2+\norm{e_{k-1}}^2\bigr).
\label{eq:stoch-dualinc}
\end{equation}
Here we used $(a+b+c+d)^2\le4(a^2+b^2+c^2+d^2)$. Set \(C_e^\star:=2/\beta+C_{\lambda,e}/\rho\).  Combining the stochastic
primal bound, \eqref{eq:dualincrease}, and \eqref{eq:stoch-dualinc} with the
memory term in \eqref{eq:potential}, and then using \eqref{eq:safe}, gives the following bound for any $k\ge1$ on a sample path whose iterates remain in the regularity region up to iteration $k+1$:
\begin{equation}
    \Phi_{k+1}
    \le \Phi_k-\frac{\beta}{4}\norm{p_k}^2
       +C_e^\star
        \big(\norm{e_k}^2+\norm{e_{k-1}}^2\big).
    \label{eq:stoch-lyap}
\end{equation}
If, almost surely, the stochastic iterates remain in the regularity region up to iteration $K$, Lemma~\ref{lem:invariant} gives $\norm{\lambda_k}\le\Lambda$ almost surely for $1\le k\le K$. The completion-of-the-square argument in \eqref{eq:lowerphi} then yields $\Phi_k\ge\underline\Phi$ almost surely over the same range.

Thus, \eqref{eq:spidersum} controls the accumulated estimator error by the
step energy, whereas \eqref{eq:stoch-lyap} controls the step energy by the
estimator error.  To close these two estimates, for \(K\ge2\), define
$$
P_K:=\sum_{k=0}^{K-1}\E\norm{p_k}^2,
\qquad
E_K:=\sum_{k=0}^{K-1}\E\norm{e_k}^2,
$$
and set
$$
D_0:=\Delta^2+\frac{4(\overline\Phi_1-\underline\Phi)}{\beta},
\qquad
D_1:=\frac{8C_e^\star}{\beta}.
$$

\begin{lemma}
\label{lem:coupled}
\leanDeclTag{Lemma_3_4.lean}{2219}
\leanDeclTagExtra{Lemma_3_4.lean}{2241}
\leanDeclTagEnd
Under Assumptions~\ref{ass:smooth} and~\ref{ass:safe}, let $K\ge2$ and suppose that, almost surely, the stochastic NR-LALM iterates remain in the regularity region up to iteration $K$. Then
$$
P_K\le D_0+D_1E_K.
$$
If, in addition, Assumption~\ref{ass:stoch} holds and the projected SPIDER estimator is used with \(b\ge2D_1QL_s^2\), then
$$
\frac{E_K}{K}
\le\frac{2\sigma_f^2}{B}+\frac{D_0}{D_1K},
\qquad
\frac{P_K}{K}
\le\frac{2D_1\sigma_f^2}{B}+\frac{2D_0}{K}.
$$
\end{lemma}

\begin{proof}
The bounds on \(x_1,c(x_1),p_0\) give
\(\Phi_1\le\overline\Phi_1\).  Summing \eqref{eq:stoch-lyap}, taking
expectations, and using \eqref{eq:lowerphi} gives
$$
    \frac{\beta}{4}\sum_{k=1}^{K-1}\E\norm{p_k}^2
    \le \overline\Phi_1-\underline\Phi
       +2C_e^\star\sum_{k=0}^{K-1}\E\norm{e_k}^2.
$$
Adding \(\E\norm{p_0}^2\le\Delta^2\) proves
\(P_K\le D_0+D_1E_K\).  If the projected SPIDER estimator is used, then,
with \(\alpha=QL_s^2/b\), \eqref{eq:spidersum} and the preceding bound give
\begin{align*}
 E_K&\le K\sigma_f^2/B+\alpha P_K,\qquad
 P_K\le D_0+D_1E_K,\\
 E_K&\le2K\sigma_f^2/B+D_0/D_1,\qquad
 P_K\le2D_1K\sigma_f^2/B+2D_0,
\end{align*}
where \(b\ge2D_1QL_s^2\) gives \(\alpha D_1\le1/2\).
Division by \(K\) proves the two asserted average bounds.
\end{proof}

The stochastic residual bound additionally contains the estimator error.
Define
$$
    C_R^{\mathrm{st}}
    :=\max\left\{
       2C_s^2+\frac{C_{\lambda,p}}{\rho^2},
       2+\frac{C_{\lambda,e}}{\rho^2}
    \right\}.
$$

\begin{lemma}
\label{lem:stochres}
\leanDeclTag{Lemma_3_5.lean}{1457}
\leanDeclTagEnd
Under Assumptions~\ref{ass:smooth} and~\ref{ass:safe}, let $N\ge2$ and suppose that, almost surely, the stochastic NR-LALM iterates remain in the regularity region up to iteration $N$. Then, almost surely, for every
\(1\le k\le N-1\),
\begin{equation}
    \mathcal{R}_{k+1}^2
    \le C_R^{\mathrm{st}}\Big(
       \norm{p_k}^2+\norm{p_{k-1}}^2+\norm{e_k}^2+\norm{e_{k-1}}^2\Big).
    \label{eq:stochres}
\end{equation}
\end{lemma}

\begin{proof}
Since \(v_k=\grad f(x_k)+e_k\), \eqref{eq:stoch-identity} gives
\begin{align*}
&\grad f(x_{k+1})+\grad c(x_{k+1})\lambda_{k+1}\\
&=-\beta p_k+\rho\grad c(x_k)d_k-e_k
  +\grad f(x_{k+1})-\grad f(x_k)\\
&\quad+\big(\grad c(x_{k+1})-\grad c(x_k)\big)\lambda_{k+1}.
\end{align*}
The estimates used in \eqref{eq:statbound} therefore give
$$
    \norm{\grad f(x_{k+1})
       +\grad c(x_{k+1})\lambda_{k+1}}
    \le C_s\norm{p_k}+\norm{e_k}.
$$
Square this bound and combine it with
\(c(x_{k+1})=(\lambda_{k+1}-\lambda_k)/\rho\) and
\eqref{eq:stoch-dualinc}; the definition of \(C_R^{\mathrm{st}}\) gives
\eqref{eq:stochres}.
\end{proof}

For an integer horizon $K\ge2$, choose
\begin{equation}
    B=K,
    \qquad Q=\lceil\sqrt K\rceil,
    \qquad
    b=\max\left\{1,
       \left\lceil 2D_1L_s^2Q\right\rceil\right\},
    \label{eq:spiderparams}
\end{equation}
and define the horizon- and accuracy-independent constants
$$
\Gamma_e:=2\sigma_f^2+\frac{D_0}{D_1},
\qquad
\Gamma_p:=2D_1\sigma_f^2+2D_0,
\qquad
C_{\mathrm{st}}
:=2C_R^{\mathrm{st}}(\Gamma_p+\Gamma_e).
$$
The choice of $b$ satisfies the condition in Lemma~\ref{lem:coupled}, which gives $E_K\le\Gamma_e$ and $P_K\le\Gamma_p$, uniformly in $K$.  Together with Lemma~\ref{lem:stochres}, these bounds yield the following direct
complexity result.  Its almost-sure containment condition is automatic when $\mathcal U=\R^n$; under local regularity, the stopped-process and safeguarded-restart analysis below provides the corresponding unconditional guarantee under the stated localization conditions.

\begin{theorem}
\label{thm:stochastic}
\leanDeclTag{Theorem_3_6.lean}{309}
\leanDeclTagExtra{Theorem_3_6.lean}{731}
\leanDeclTagEnd
Let Assumptions~\ref{ass:smooth}, \ref{ass:stoch}, and~\ref{ass:safe} hold
and fix \(K\ge2\).  Run
stochastic NR-LALM with the projected SPIDER estimator and the parameters in
\eqref{eq:spiderparams}, and suppose that, almost surely, its iterates remain in the regularity region up to iteration $K$. Let \(\widehat{k}\) be independent of the oracle
samples and uniform on \(\{1,\ldots,K-1\}\).  Then
\begin{equation}
    \E[\mathcal{R}_{\widehat{k}+1}^2]
    \le \frac{C_{\mathrm{st}}}{K-1},
    \label{eq:stochrate}
\end{equation}
where the expectation is over both the oracle samples and \(\widehat{k}\), and
\(C_{\mathrm{st}}\) is independent of \(K\) and \(\eps\).  Consequently, if
\(K-1\ge C_{\mathrm{st}}\eps^{-2}\), then
\((x_{\widehat{k}+1},\lambda_{\widehat{k}+1})\) is a stochastic
\(\eps\)-KKT pair.  Hence the iteration, deterministic constraint/Jacobian
evaluation, and exact linear-system solve complexities are
\(\cO(\eps^{-2})\), whereas the stochastic-gradient complexity is
\(\cO(\eps^{-3})\).
\end{theorem}

\begin{proof}
By Lemma~\ref{lem:coupled} and \eqref{eq:spiderparams},
\(P_K\le\Gamma_p\) and \(E_K\le\Gamma_e\).
By Lemma~\ref{lem:stochres} and the independence and uniformity of \(\widehat{k}\),
\begin{equation*}
    \E[\mathcal R_{\widehat{k}+1}^2]
    =\frac1{K-1}\sum_{k=1}^{K-1}\E[\mathcal R_{k+1}^2]
    \le\frac{2C_R^{\mathrm{st}}(P_K+E_K)}{K-1}
    \le\frac{C_{\mathrm{st}}}{K-1}.
\end{equation*}
If \(K-1\ge C_{\mathrm{st}}\eps^{-2}\), then the preceding bound gives \(\E[\mathcal R_{\widehat{k}+1}^2]\le\eps^2\), so
Definition~\ref{def:stoch-eps-kkt} gives the asserted stochastic \(\eps\)-KKT pair. There are at most \(\lceil K/Q\rceil\) checkpoint
batches, each of size \(B\).  At each noncheckpoint iteration, the same batch of \(b\) samples is evaluated at both \(x_k\) and \(x_{k-1}\), resulting in \(2b\) stochastic gradient evaluations.  Therefore
$$
    N_{\mathrm{grad}}
    \le \left\lceil\frac KQ\right\rceil B+2Kb
    =\cO(K^{3/2})=\cO(\eps^{-3}).
$$
The deterministic evaluation and solve counts follow from the per-iteration cost stated after Algorithm~\ref{alg:lalm}.
\end{proof}

The direct theorem still requires every sampled trajectory to remain in \(\mathcal U\).  Under local assumptions, an expected Lyapunov bound does not provide this pathwise guarantee.  We therefore stop a run when it first leaves a prescribed set whose neighborhood lies in \(\mathcal U\).  Theorem \ref{thm:stochlocalization} bounds the exit probability over a finite horizon, and restarting failed runs yields the unconditional expected bound in Corollary~\ref{cor:restart}.  This construction uses an exact membership test and is primarily a theoretical localization device.

Fix \(K\ge2\) and \(\delta\in(0,1)\), and use
\eqref{eq:spiderparams}.  Define
$$
    H_{\mathrm{st}}(\delta)
    :=H_{\mathrm{det}}+\frac{2C_e^\star\Gamma_e}{\delta},
    \qquad
    \mathcal S_{\mathrm{st}}(\delta)
    :=\left\{x:f(x)\le H_{\mathrm{st}}(\delta),
       \norm{c(x)}\le\frac{2\Lambda}{\rho}\right\}.
$$
Suppose there is a known Borel set $\mathcal X_\delta$ with an exact
membership test such that
\begin{equation}
    x_0\in\mathcal X_\delta,
    \qquad
    \mathcal S_{\mathrm{st}}(\delta)
    \subseteq\mathcal X_\delta,
    \qquad
    \mathcal N_{\Delta}
       \big(\mathcal X_\delta\big)
    \subset\mathcal U.
    \label{eq:stochbuffer}
\end{equation}
The buffer gives \(x_0\in\mathcal U\), hence \(f(x_0)\ge\underline f\),
\(D_0,\Gamma_e>0\), and \(H_{\mathrm{st}}(\delta)\ge H_{\mathrm{det}}\).
Together with \(\norm{p_0}\le\Delta\), it places
\([x_0,x_1]\subset\mathcal N_\Delta(\mathcal X_\delta)\subset\mathcal U\)
and initializes the stopped recursion.
Consider the stochastic NR-LALM with projected SPIDER, constructed recursively
while its current iterate lies in \(\mathcal X_\delta\), and let
\(\tau_{\rm ex}:=\inf\{k\ge1:x_k\notin\mathcal X_\delta\}\), with
\(\tau_{\rm ex}=\infty\) if no exit occurs.

\begin{theorem}
\label{thm:stochlocalization}
\leanDeclTag{Theorem_3_7.lean}{5068}
\leanDeclTagExtra{Theorem_3_7.lean}{5558}
\leanDeclTagEnd
Under the preceding setup, including \(x_0\in\mathcal X_\delta\), suppose that
Assumptions~\ref{ass:smooth}, \ref{ass:stoch}, and~\ref{ass:safe} hold.  Then
\begin{equation}
    \mathbb P(\tau_{\rm ex}\le K)\le\delta.
    \label{eq:exitprob}
\end{equation}
If \(\widehat{k}\) is independent of the oracle samples and uniform on
$\{1,\ldots,K-1\}$, then
\begin{equation}
    \E\left[\mathcal R_{\widehat{k}+1}^2\mid\tau_{\rm ex}>K\right]
    \le\frac{C_{\mathrm{st}}}
             {(1-\delta)(K-1)}.
    \label{eq:conditionalstochrate}
\end{equation}
\end{theorem}

\begin{proof}
After an exit, no further oracle calls are made.  To write the stopped sums on a fixed horizon, set \(x_k=x_{\tau_{\rm ex}}\),
\(\lambda_k=\lambda_{\tau_{\rm ex}}\), and \(p_k=e_k=0\) for
\(k\ge\tau_{\rm ex}\); the estimators \(h_k\) and \(v_k\) are not evaluated. Also set
\(I_k=\mathbf1_{\{k<\tau_{\rm ex}\}}\) and
$$
 P_K^{\rm stop}:=\sum_{k=0}^{K-1}\E[I_k\norm{p_k}^2],
 \qquad
 E_K^{\rm stop}:=\sum_{k=0}^{K-1}\E[I_k\norm{e_k}^2].
$$
Before exit, \(x_k\in\mathcal X_\delta\); the buffer
\eqref{eq:stochbuffer} ensures $[x_k,x_{k+1}]\subset\mathcal U$ whenever $k<\tau_{\rm ex}$.  Moreover, \(I_k\) is \(\mathcal G_k\)-measurable and \(I_k\le I_{k-1}\) for $k\ge1$.  Repeating the epochwise proof of Lemma~\ref{lem:spider} with these indicators gives the first
inequality below.  The active indices form an initial block, so the Lyapunov sum telescopes to
\(\Phi_1-\Phi_{\min\{K,\tau_{\rm ex}\}}\); repeating the proof of
Lemma~\ref{lem:coupled} gives the second:
$$
 E_K^{\rm stop}\le\frac{K\sigma_f^2}{B}
 +\frac{QL_s^2}{b}P_K^{\rm stop},
 \qquad
 P_K^{\rm stop}\le D_0+D_1E_K^{\rm stop}.
$$
Thus \(B=K\) and \eqref{eq:spiderparams} imply
\begin{equation}
 E_K^{\rm stop}\le\Gamma_e,\qquad
 P_K^{\rm stop}\le\Gamma_p.
 \label{eq:stoppedbounds}
\end{equation}

The base-step estimates give \(f(x_1)\le H_{\mathrm{det}}\) and
\(\rho\norm{c(x_1)}\le2\Lambda\), so
\(x_1\in\mathcal S_{\mathrm{st}}(\delta)\).  If \(\tau_{\rm ex}\le K\),
the stopped Lyapunov inequality and completion of the square give
$$
 f(x_{\tau_{\rm ex}})
 \le H_{\mathrm{det}}+
 2C_e^\star\sum_{k=0}^{K-1}I_k\norm{e_k}^2.
$$
Lemma~\ref{lem:invariant}, applied before exit, also gives
\(\rho\norm{c(x_{\tau_{\rm ex}})}\le2\Lambda\).  Since
\(\mathcal S_{\mathrm{st}}(\delta)\subseteq\mathcal X_\delta\), exit
therefore implies
\(f(x_{\tau_{\rm ex}})>H_{\mathrm{st}}(\delta)\).  Hence
\(\sum_kI_k\norm{e_k}^2>\Gamma_e/\delta\), and Markov's inequality
with \eqref{eq:stoppedbounds} proves \eqref{eq:exitprob}.

On \(\{\tau_{\rm ex}>K\}\), Lemma~\ref{lem:stochres} and
\eqref{eq:stoppedbounds} give
$$
 \E\!\left[
 \mathbf1_{\{\tau_{\rm ex}>K\}}\mathcal R_{\widehat{k}+1}^2\right]
 \le\frac{2C_R^{\mathrm{st}}
 (P_K^{\rm stop}+E_K^{\rm stop})}{K-1}
 \le\frac{C_{\mathrm{st}}}{K-1}.
$$
Dividing by
\(\mathbb P(\tau_{\rm ex}>K)\ge1-\delta\) proves
\eqref{eq:conditionalstochrate}.
\end{proof}

For fixed \(\delta\), the level \(H_{\mathrm{st}}(\delta)\) is independent of \(K\) and \(\eps\).  The set \(\mathcal X_\delta\) may, for example, be a known box or ball; under global regularity one may take \(\mathcal X_\delta=\mathbb R^n\).  To remove the conditioning in Theorem~\ref{thm:stochlocalization}, restart from $(x_0,\lambda_0)$ after each exit, using fresh independent oracle samples.  From the first attempt that completes $K$ steps, draw $\widehat{k}$ independently and uniformly from $\{1,\ldots,K-1\}$ and return the primal--dual pair $(x_{\widehat{k}+1},\lambda_{\widehat{k}+1})$.

\begin{corollary}
\label{cor:restart}
\leanDeclTag{Corollary_3_8.lean}{822}
\leanDeclTagExtra{Corollary_3_8.lean}{2059}
\leanDeclTagExtra{Corollary_3_8.lean}{2316}
\leanDeclTagEnd
Under the conditions of Theorem~\ref{thm:stochlocalization}, let \(T\) be the
number of attempts and
\(N_{\mathrm{grad}}\) the total stochastic-gradient count.  The procedure
terminates almost surely and
\begin{equation}
    \E[T]\le\frac{1}{1-\delta},\quad
    \E[\mathcal R_{\widehat{k}+1}^2]
    \le\frac{C_{\mathrm{st}}}{(1-\delta)(K-1)},\quad
    \E[N_{\mathrm{grad}}]
    \le\frac{\lceil K/Q\rceil B+2Kb}{1-\delta}.
    \label{eq:restartbounds}
\end{equation}
The expected numbers of deterministic constraint/Jacobian evaluations,
exact linear-system solves, and localization-set membership tests are each
$\cO(K/(1-\delta))$.
\end{corollary}

\begin{proof}
Let \(\pi_K:=\mathbb P(\tau_{\rm ex}>K)\).  By
Theorem~\ref{thm:stochlocalization}, \(\pi_K\ge1-\delta\).  Independent fresh
attempts make \(T\) geometric, so
\(\E[T]=1/\pi_K\le1/(1-\delta)\); the accepted attempt has the conditional
law in \eqref{eq:conditionalstochrate}.  Finally, pathwise,
$$
 N_{\mathrm{grad}}\le
 T(\lceil K/Q\rceil B+2Kb),
$$
and the other work counts are bounded by constant multiples of \(KT\).
Taking expectations proves the claims.
\end{proof}

For any fixed $\delta\in(0,1)$, choose
$K=\left\lceil C_{\mathrm{st}}/[(1-\delta)\eps^2]\right\rceil+1$.
Then Corollary~\ref{cor:restart} gives a stochastic $\eps$-KKT pair with $\cO(\eps^{-3})$ expected stochastic gradient evaluations and $\cO(\eps^{-2})$ expected deterministic evaluations and exact linear-system solves. Thus, under the additional exact-membership assumption, restarting removes the pathwise containment condition without changing the complexity exponents.

\section{A second-order correction}
\label{sec:correction}

We append an optional minimum-norm second-order correction to the NR-LALM primal step.  Given $p_k$, let $z_k=x_k+p_k$ and define
\begin{equation}
\begin{aligned}
r_k&=c(z_k)-c(x_k)-\grad c(x_k)^\top p_k,\\
q_k&=-\grad c(z_k)
\big(\grad c(z_k)^\top\grad c(z_k)\big)^{-1}r_k,\\
x_{k+1}&=z_k+q_k,
\end{aligned}
\label{eq:correctionstep}
\end{equation}
The multiplier update remains $\lambda_{k+1}=\lambda_k+\rho c(x_{k+1})$.  We call the resulting variant NR-LALM+SOC.  Uniform LICQ makes $q_k$ the minimum-norm solution of $\grad c(z_k)^\top q_k=-r_k$.  The local correction requires one additional evaluation of $c$ and $\grad c$ and one $m\times m$ linear solve; it is not a
globalization procedure. For NR-LALM+SOC, we say that the iterates remain in the regularity region up to iteration $N$ if $[x_k,z_k]\cup[z_k,x_{k+1}]\subset\mathcal U$ for $0\le k\le N-1$.

Define $d_k^{\mathrm{cor}} :=c(x_{k+1})-c(x_k)-\grad c(x_k)^\top p_k$.  Taylor's theorem gives $\norm{r_k}\le\kappa_{\mathrm{lin}}\norm{p_k}^2$, and uniform LICQ gives $\norm{q_k}\le\sigma^{-1}\norm{r_k}$.  Since
$\grad c(z_k)^\top q_k=-r_k$, we have $d_k^{\mathrm{cor}}=c(z_k+q_k)-c(z_k)-\grad c(z_k)^\top q_k$; hence
$d_k^{\mathrm{cor}}$ is the Taylor remainder along the correction step.  Set $\kappa_q=L_c/(2\sigma)$, $\kappa_d=L_c^3/(8\sigma^2)$, and $\chi^{\mathrm{cor}}=1+\kappa_q\Delta$. Whenever $[x_k,z_k]\cup[z_k,x_{k+1}]\subset\mathcal U$ and $\norm{p_k}\le\Delta$, these relations yield
\begin{equation}
\begin{aligned}
 \norm{r_k}&\le\kappa_{\mathrm{lin}}\norm{p_k}^2,&
 \norm{q_k}&\le\kappa_q\norm{p_k}^2,\\
 \norm{d_k^{\mathrm{cor}}}
 &\le\kappa_d\norm{p_k}^4
 \le\kappa_d\Delta^2\norm{p_k}^2,&
 \norm{x_{k+1}-x_k}
 &\le\chi^{\mathrm{cor}}\norm{p_k}.
\end{aligned}
\label{eq:corrected-error-bounds}
\end{equation}
The primal optimality condition and nonlinear-residual update consequently give
$\grad f(x_k)+\grad c(x_k)\lambda_{k+1}+\beta p_k
=\rho\grad c(x_k)d_k^{\mathrm{cor}}$.  For the stochastic method, the same
identity holds with $\grad f(x_k)$ replaced by $v_k$.  Table~\ref{tab:corrected-constants}
lists the corresponding constant replacements.

\begin{table}[!t]
\caption{Constants used in the analysis of NR-LALM+SOC.}
\label{tab:corrected-constants}
\centering
\scriptsize
\renewcommand{\arraystretch}{0.94}
\begin{tabularx}{\textwidth}{@{}>{\raggedright\arraybackslash}p{0.25\textwidth}>{\raggedright\arraybackslash}X@{}}
\toprule
NR-LALM & NR-LALM+SOC\\
\midrule
\(\kappa_{\mathrm{lin}}\), \(1\)
& \(\kappa_d\Delta^2\),
  \(\chi^{\mathrm{cor}}=1+\kappa_q\Delta\)\\
\(C_{\mathrm{mod}}\)
& \(\displaystyle C_{\mathrm{mod}}^{\mathrm{cor}}
 =G\kappa_q+\frac{L_f}{2}(\chi^{\mathrm{cor}})^2
 +\kappa_d\Delta^2(3\Lambda+\rho M\Delta)
 +\frac{\rho}{2}\kappa_d^2\Delta^6\)\\
\(A_p\)
& \(\displaystyle A_p^{\mathrm{cor}}
 =\beta+\rho M\kappa_d\Delta^3\)\\
\(B_p\)
& \(\displaystyle B_p^{\mathrm{cor}}
 =A_p^{\mathrm{cor}}+(L_f+L_c\Lambda)\chi^{\mathrm{cor}}\)\\
\(C_{\lambda,p}\)
& \(\displaystyle C_{\lambda,p}^{\mathrm{cor}}
 =\frac{4}{\sigma^2}
 \max\{(A_p^{\mathrm{cor}})^2,(B_p^{\mathrm{cor}})^2\}\)\\
\(C_s\)
& \(\displaystyle C_s^{\mathrm{cor}}
 =\beta+\rho M\kappa_d\Delta^3
 +(L_f+L_c\Lambda)\chi^{\mathrm{cor}}\)\\
\(\overline\Phi_1\)
& \(\displaystyle \overline\Phi_1^{\mathrm{cor}}
 =f(x_0)+G\chi^{\mathrm{cor}}\Delta+\frac{4\Lambda^2}{\rho}
 +\frac{C_{\lambda,p}^{\mathrm{cor}}}{\rho}\Delta^2\)\\
\(H_{\mathrm{det}}\)
& \(\displaystyle H_{\mathrm{det}}^{\mathrm{cor}}
 =\overline\Phi_1^{\mathrm{cor}}+\frac{\Lambda^2}{2\rho}\)\\
\(L_s^2\) in SPIDER displacement bounds
& \(L_s^2(\chi^{\mathrm{cor}})^2\)\\
\bottomrule
\end{tabularx}
\end{table}

For the deterministic analysis of NR-LALM+SOC, use the four inequalities in \eqref{eq:safe} with
$\kappa_{\mathrm{lin}}$, $C_{\mathrm{mod}}$, and $C_{\lambda,p}$
replaced by $\kappa_d\Delta^2$, $C_{\mathrm{mod}}^{\mathrm{cor}}$, and $C_{\lambda,p}^{\mathrm{cor}}$, respectively.  The initialization conditions are unchanged.  The deterministic localization set and buffer become
\begin{equation}
 \mathcal S_{\mathrm{det}}^{\mathrm{cor}}
 :=\left\{x:f(x)\le H_{\mathrm{det}}^{\mathrm{cor}},
          \norm{c(x)}\le2\Lambda/\rho\right\},\qquad
 \mathcal N_{\chi^{\mathrm{cor}}\Delta}
 \bigl(\mathcal S_{\mathrm{det}}^{\mathrm{cor}}\bigr)\subset\mathcal U.
\label{eq:corrected-det-buffer}
\end{equation}
Indeed, if $x_k\in\mathcal S_{\mathrm{det}}^{\mathrm{cor}}$ and
$\norm{p_k}\le\Delta$, the displacement bounds in
\eqref{eq:corrected-error-bounds} keep both
$[x_k,z_k]$ and $[z_k,x_{k+1}]$ in this neighborhood.

For the stochastic analysis,
$\norm{x_{k+1}-x_k}\le\chi^{\mathrm{cor}}\norm{p_k}$ changes the SPIDER
displacement factor from $L_s^2$ to
$L_s^2(\chi^{\mathrm{cor}})^2$.  The quantities
$C_e^\star$, $C_{\lambda,e}$, and $D_1$ remain unchanged, while
$D_0^{\mathrm{cor}}$, $\Gamma_e^{\mathrm{cor}}$,
$\Gamma_p^{\mathrm{cor}}$, $C_R^{\mathrm{st,cor}}$, and
$C_{\mathrm{st}}^{\mathrm{cor}}$ are defined by the corresponding formulas
in Section~\ref{sec:stochastic} with the corrected constants.  Set
$H_{\mathrm{st}}^{\mathrm{cor}}(\delta)
:=H_{\mathrm{det}}^{\mathrm{cor}}
+2C_e^\star\Gamma_e^{\mathrm{cor}}/\delta$ and
$\mathcal S_{\mathrm{st}}^{\mathrm{cor}}(\delta)
:=\{x:f(x)\le H_{\mathrm{st}}^{\mathrm{cor}}(\delta),
\ \norm{c(x)}\le2\Lambda/\rho\}$.  The stopped-process argument uses
\begin{equation}
 x_0\in\mathcal X_\delta,
 \qquad
 \mathcal S_{\mathrm{st}}^{\mathrm{cor}}(\delta)
 \subseteq\mathcal X_\delta,
 \qquad
 \mathcal N_{\chi^{\mathrm{cor}}\Delta}(\mathcal X_\delta)
 \subset\mathcal U.
\label{eq:corrected-stoch-buffer}
\end{equation}
together with $B=K$, $Q=\lceil\sqrt K\rceil$, and
$b=\max\{1,\lceil2D_1L_s^2(\chi^{\mathrm{cor}})^2Q\rceil\}$.

The correction reduces the linearization-error coefficient from
$\kappa_{\mathrm{lin}}$ to $\kappa_d\Delta^2$, but increases the displacement factor from $1$ to $\chi^{\mathrm{cor}}$.  The next result compares the combined effect of these changes on the deterministic sufficient conditions.

\begin{proposition}
\label{prop:parameter-region}
\leanDeclTag{Proposition_4_1.lean}{322}
\leanDeclTagExtra{Proposition_4_1.lean}{493}
\leanDeclTagEnd
Suppose Assumption~\ref{ass:smooth} holds with \(m\ge1\), and use the same regularity constants for both methods.  Every positive tuple \((\Delta,\beta,\rho,\Lambda)\) satisfying \eqref{eq:safe} for NR-LALM also satisfies the corrected inequalities for NR-LALM+SOC.  Consequently, under global regularity, any parameters and initial pair certified for NR-LALM are also certified for NR-LALM+SOC.  This inclusion can be strict: there exist a globally regular scalar problem and fixed \((\beta,\rho,x_0,\lambda_0)\) for which the NR-LALM+SOC conditions hold, whereas the NR-LALM conditions admit no positive auxiliary pair
\((\widehat\Delta,\widehat\Lambda)\).
\end{proposition}

\begin{proof}
Set \(\vartheta=\kappa_q\Delta\), \(\tau_\rho=\rho/\beta\), and
\(\mu=M\kappa_{\mathrm{lin}}\Delta=M\sigma\vartheta\).  The fourth inequality
in \eqref{eq:safe} gives
\(\tau_\rho\ge(32/\sigma^2)(1+\tau_\rho\mu)^2\).
Since \(\sup_{t>0}t/(1+t\mu)^2=1/(4\mu)\) and Assumption~\ref{ass:smooth}
gives \(M\ge\sigma\) when \(m\ge1\),
\(\mu\le\sigma^2/128\), \(\vartheta\le1/128\), and
\(\kappa_d\Delta^2=\vartheta^2\kappa_{\mathrm{lin}}\).  Thus the first corrected inequality is weaker, and the second is unchanged.

Let \(T_d=\kappa_{\mathrm{lin}}(3\Lambda+\rho M\Delta)\) and
\(W_d=(\rho/2)\kappa_{\mathrm{lin}}^2\Delta^2\).  The first and third inequalities in \eqref{eq:safe} give, respectively,
\(T_d\ge3G\kappa_q+3\beta\vartheta\) and \(L_f\le3\beta/4\).  Hence
\begin{equation*}
 C_{\mathrm{mod}}-C_{\mathrm{mod}}^{\mathrm{cor}}
 =(1-\vartheta^2)T_d+(1-\vartheta^4)W_d-G\kappa_q
   -\frac{L_f}{2}(2\vartheta+\vartheta^2)
 \ge0,
\end{equation*}
where the last inequality uses \(3(1-\vartheta^2)\ge1\) and
\(3(1-\vartheta^2)\ge(3/8)(2+\vartheta)\).

Finally, put \(D_d=\rho M\kappa_{\mathrm{lin}}\Delta\) and
\(H_d=L_f+L_c\Lambda\).  The third and fourth inequalities in
\eqref{eq:safe} imply
\(H_d\le\beta\), \(\rho\ge32\beta/\sigma^2\), and therefore
\((1-\vartheta^2)\rho M\sigma>H_d\).  Consequently,
$$
 A_p^{\mathrm{cor}}\le A_p,\qquad
 B_p-B_p^{\mathrm{cor}}
 =\vartheta[(1-\vartheta^2)\rho M\sigma-H_d]\ge0.
$$
Thus \(C_{\lambda,p}^{\mathrm{cor}}\le
C_{\lambda,p}\).  Together with the preceding comparisons, this proves that all four corrected inequalities hold for the same tuple.

For strictness, take the globally regular scalar problem
\(f(x)=\sin x\), \(c(x)=2x-\cos x\), and \(\mathcal U=\mathbb R\), with
the sharp constants \(G=L_f=L_c=\sigma=1\), \(M=3\), and
\(\underline f=-1\).  Since \(c'(x)\ge1\) and
\(c(0)<0<c(1/2)\), \(c\) has a unique root \(x_0\).  Choose
\(\lambda_0=\Lambda=4\), \(\Delta=1/20\), \(\beta=50\), and \(\rho=2200\).
Here \(\kappa_d\Delta^2=1/3200\),
\(\chi^{\mathrm{cor}}=41/40\),
\(A_p^{\mathrm{cor}}=16033/320\), and
\(B_p^{\mathrm{cor}}=17673/320<56\).  Direct substitution gives
\begin{align*}
 \frac{G+\beta\Delta+\rho M\kappa_d\Delta^4}{\sigma}
 &=\frac{22433}{6400}<4,\qquad
 \frac{G}{\beta}+\frac{3M\Lambda}{\beta+\rho\sigma^2}
 =\frac9{250}<\frac1{20},\\
 C_{\mathrm{mod}}^{\mathrm{cor}}&<\frac65<\frac{3\beta}{8},\qquad
 \frac{8C_{\lambda,p}^{\mathrm{cor}}}{\beta}<2008<\rho.
\end{align*}
Thus the corrected parameter and initialization conditions hold, and
localization is automatic.  Conversely, if the NR-LALM sufficient conditions
held for the same \((\beta,\rho,x_0,\lambda_0)\), they would require
\(\widehat\Lambda\ge4\).  The second inequality in \eqref{eq:safe} then gives
\(\widehat\Delta\ge9/250\).  On the other hand, with
\(\widehat A_p:=\beta+\rho M\kappa_{\mathrm{lin}}\widehat\Delta\), the fourth
inequality in \eqref{eq:safe} gives
\(44\ge32(1+66\widehat\Delta)^2\), and hence
\(\widehat\Delta\le
(\sqrt{11/8}-1)/66<1/330\).
Since \(9/250>1/330\), no such auxiliary pair exists.
\end{proof}

Proposition~\ref{prop:parameter-region} compares only the sufficient
conditions derived above, not the actual convergence domains of the two methods.  We now transfer the convergence and complexity results of Sections~\ref{sec:deterministic} and~\ref{sec:stochastic} to NR-LALM+SOC.

\begin{corollary}
\label{cor:corrected-transfer}
Under Assumption~\ref{ass:smooth}, suppose the parameters satisfy
\eqref{eq:safe} with  $\kappa_{\mathrm{lin}}$, $C_{\mathrm{mod}}$, and $C_{\lambda,p}$ replaced by $\kappa_d\Delta^2$, $C_{\mathrm{mod}}^{\mathrm{cor}}$, and
$C_{\lambda,p}^{\mathrm{cor}}$, respectively, and the initialization
conditions in Assumption~\ref{ass:safe} hold.  Use the other constants from
Table~\ref{tab:corrected-constants}.

\medskip
\noindent\textup{(i)}
\leanDeclTag{Corollary_4_2.lean}{1836}
\leanDeclTagExtra{Corollary_4_2.lean}{1865}
\leanDeclTagEnd
If
\eqref{eq:corrected-det-buffer} holds, deterministic NR-LALM+SOC has
\(\mathcal O(\varepsilon^{-2})\) iteration, first-order-oracle, primal-solve,
and correction-solve complexities.

\medskip
\noindent\textup{(ii)}
\leanDeclTag{Corollary_4_2.lean}{4788}
\leanDeclTagExtra{Corollary_4_2.lean}{7246}
\leanDeclTagEnd
If, in addition,
Assumption~\ref{ass:stoch} holds, stochastic NR-LALM+SOC with the corrected projected-SPIDER parameters has \(\mathcal O(\varepsilon^{-3})\) stochastic-gradient complexity and \(\mathcal O(\varepsilon^{-2})\) constraint/Jacobian evaluation and solve complexities under almost-sure admissibility.  Under \eqref{eq:corrected-stoch-buffer}, safeguarded restarts
give the same orders in expectation.

\medskip
\noindent\textup{(iii)}
\leanDeclTag{Corollary_4_2.lean}{8415}
\leanDeclTagExtra{Corollary_4_2.lean}{8674}
\leanDeclTagEnd
For deterministic NR-LALM+SOC, if \eqref{eq:corrected-det-buffer} holds, \(\mathcal S_{\mathrm{det}}^{\mathrm{cor}}\) is compact, and the function \(\mathcal E\) defined in \eqref{eq:KLenergy} has the KL property at every point of the cluster set of \((x_k,\lambda_k,p_{k-1})\), then the primal--dual
trajectory has finite length and converges to a KKT pair.
\end{corollary}

\begin{proof}
After the tabulated substitutions, the corrected identity and
\eqref{eq:corrected-error-bounds} give, whenever $[x_k,z_k]\cup[z_k,x_{k+1}]\subset\mathcal U$,
\begin{align*}
 \cL_\rho(x_{k+1},\lambda_k)-\cL_\rho(x_k,\lambda_k)
 &\le(-\beta+C_{\mathrm{mod}}^{\mathrm{cor}})\norm{p_k}^2
 \le-\frac{5\beta}{8}\norm{p_k}^2,\\
 \sigma\norm{\lambda_{k+1}-\lambda_k}
 &\le A_p^{\mathrm{cor}}\norm{p_k}
 +B_p^{\mathrm{cor}}\norm{p_{k-1}}.
\end{align*}
Applying the deterministic estimates of Section~\ref{sec:deterministic} with these constants proves \textup{(i)}.  For the stochastic method, replacing $L_s^2$ by $L_s^2(\chi^{\mathrm{cor}})^2$ in the SPIDER estimate and using the
corrected constants gives the coupling, residual, and restart bounds of Section~\ref{sec:stochastic}, proving \textup{(ii)}.

For \textup{(iii)}, let
\(u_k^{\mathrm{cor}}=(x_k,\lambda_k,p_{k-1})\),
\(\mathcal E_k^{\mathrm{cor}}=\mathcal E(u_k^{\mathrm{cor}})\), and
\(\ell_k^{\mathrm{cor}}=\norm{p_{k-1}}+
\norm{\lambda_k-\lambda_{k-1}}\).  The corrected estimates give
\begin{align*}
 \mathcal E_k^{\mathrm{cor}}-\mathcal E_{k+1}^{\mathrm{cor}}
 &\ge\left(\frac{\beta}{4}
 -\frac{C_{\lambda,p}^{\mathrm{cor}}}{\rho}\right)
 (\norm{p_k}^2+\norm{p_{k-1}}^2),\\
 \norm{\grad\mathcal E(u_k^{\mathrm{cor}})}
 &\le\left(C_s^{\mathrm{cor}}+\frac{\beta}{2}+M+\frac1\rho\right)
 \ell_k^{\mathrm{cor}} .
\end{align*}
The first coefficient is at least \(\beta/8\).  These are precisely the sufficient-decrease and relative-error estimates used in Theorem~\ref{thm:KL}; uniformized KL therefore gives
\(\sum_k(\norm{p_k}+\norm{\lambda_{k+1}-\lambda_k})<\infty\).
Since \(\norm{x_{k+1}-x_k}\le\chi^{\mathrm{cor}}\norm{p_k}\), the corrected trajectory has finite length, and the multiplier update and stationarity estimate identify its limit as a KKT pair.
\end{proof}

Since \(q_k=\mathcal O(\norm{p_k}^2)\), the correction changes neither the complexity exponents nor the first-order local iteration map.

\section{Numerical experiments}
\label{sec:numerics}

The experiments begin with a controlled example showing that the sufficient parameter conditions can hold simultaneously and confirming the quadratic and quartic constraint-model error orders. We then compare NR-LALM and NR-LALM+SOC with representative methods on high-dimensional deterministic problems under a common stopping and timing protocol.  Finally, we evaluate the stochastic variants under matched stochastic-gradient budgets.  The source code is available at \url{https://github.com/bqliu815/NR-LALM}.

\subsection{Sufficient parameter regions and constraint-linearization errors}
\label{subsec:mechanism-verification}

We illustrate the sufficient parameter conditions and the predicted orders of
the constraint-linearization errors on the following globally regular problem:
\begin{equation}
 f(x)=\sin\big((Bx)_1\big),\qquad
 c(x)=U\big(2Bx-\cos(Bx)\big),\qquad
 BB^\top=U^\top U=I_m,
 \label{eq:mechanism-problem}
\end{equation}
where $B\in\R^{m\times n}$, $U\in\R^{m\times m}$, and the cosine is
applied componentwise.  We take $(n,m)=(100,20)$.  If $r$ denotes the
unique root of $2r-\cos r=0$, then $x_0=B^\top(r\bm{1}_m)$ is feasible.
The global regularity constants may be chosen as
$G=L_f=L_c=\sigma=1$, $M=3$, and $\underline f=-1$.

Figure~\ref{fig:mechanism-verification}\textup{(a)} displays the pairs
$(\Delta,\Lambda)$ satisfying the sufficient conditions for
$\beta=3\times10^6$, $\rho=1.92\times10^8$, and
$\norm{\lambda_0}=1049$ over
$[10^{-5},0.018]\times[1049,50000]$.  Both sets are nonempty, and the
NR-LALM set is strictly contained in the NR-LALM+SOC set on the displayed
grid.  Panel \textup{(b)} plots the median base and corrected
constraint-linearization errors over $32$ unit directions and $12$
logarithmically spaced step norms in $[0.01,0.2]$.  Fitting each direction
separately gives median log--log slopes of $2.001$ and $4.022$, respectively,
with $R^2>0.9999$ for every fit.  These values agree with the second- and
fourth-order bounds in \eqref{eq:linerrorbounds} and
\eqref{eq:corrected-error-bounds}.

\begin{figure}[H]
\centering
\includegraphics[width=0.62\textwidth]
{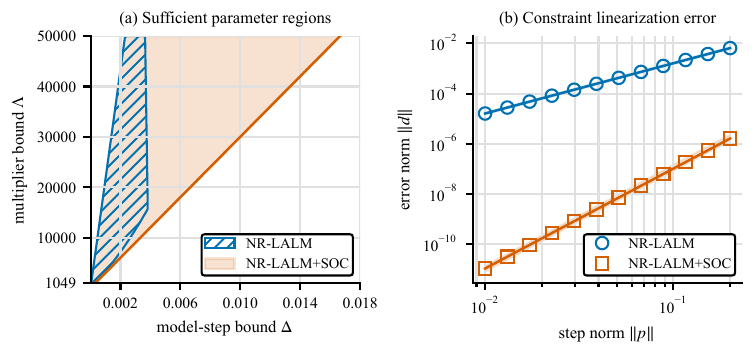}
\caption{Parameter regions and error orders.
\textup{(a)} Pairs satisfying the sufficient conditions for NR-LALM
(hatched) and NR-LALM+SOC (shaded).
\textup{(b)} Median constraint-linearization errors over $32$ directions;
bands show interquartile ranges, and solid lines show log--log fits.}
\label{fig:mechanism-verification}
\end{figure}

\FloatBarrier
\Needspace{8\baselineskip}

\subsection{Deterministic performance on high-dimensional data}
\label{subsec:deterministic-comparison}

We use the $15$ binary-classification data sets from the Library for Support Vector Machines (LIBSVM)~\cite{ChangLin2011} satisfying $n\ge2561$ and having compressed size at most $5$ GiB.  For larger data sets, we retain at most $20{,}000$ observations.  After a stratified $80/20$ split, the training feature vectors are normalized and used to define
\begin{equation}
 \min_{x\in\R^n}
 \frac{1}{N}\sum_{i=1}^N
 \log\!\left(1+\exp(-y_i d_i^\top x)\right)
 \quad\text{subject to}\quad
 Ax=b,\qquad \frac{\norm{x}^2-1}{2}=0.
 \label{eq:libsvm-constrained-logistic}
\end{equation}
The ten rows of $A$ have disjoint signed supports of cardinality $256$ and satisfy $AA^\top=I_{10}$.  We set $b=0.5e_1$ and use the common initial pair
$$
x_0=A^\top b+\sqrt{1-\norm{b}^2}\,v,\qquad \lambda_0=0,
$$
where $v$ is a unit vector in $\ker(A)$.  Thus, $x_0$ is feasible and satisfies the linear independence constraint qualification.

We compare NR-LALM, NR-LALM+SOC, L-AL~\cite{BourkhissiNecoara2025}, and the Interior Point OPTimizer (IPOPT)~\cite{WachterBiegler2006}. IPOPT uses its limited-memory Hessian approximation and the Parallel Direct Sparse Solver (PARDISO) from the Intel oneAPI Math Kernel Library~\cite{SchenkEtAl2001}. The three linearized methods use $\rho=100$, the same initial proximal coefficient $\beta_0=10$,
and the same $11$-dimensional Woodbury implementation.  These parameters were selected on separate development problems and then fixed for all reported data sets.  The method-specific acceptance tests and numerical safeguards are provided in the public code. For each method and data set, we perform eight single-thread runs on identical Intel Xeon Platinum 8558P nodes, with the execution order balanced across repeats.  Each run is limited to $200$ iterations and $1800$ seconds.  We report the median elapsed time to the first iterate satisfying $\mathcal R(x_k,\lambda_k)^2\le10^{-8}$, as measured by the common residual
evaluator.  Data loading and instance construction are excluded.  In
Table~\ref{tab:libsvm-stage-b-timing}, boldface indicates the row minimum.

\begin{table}[!t]
\centering
\caption{Median first-hit time (seconds).}
\label{tab:libsvm-stage-b-timing}
\scriptsize
\setlength{\tabcolsep}{2.8pt}
\renewcommand{\arraystretch}{0.88}
\begin{tabular}{@{}lrrrrrr@{}}
\toprule
Data set & $N$ & $n$ & NR-LALM & NR-LALM+SOC & L-AL & IPOPT \\
\midrule
avazu-app
& $16{,}000$ & $1{,}000{,}000$
& \textbf{0.71} & 0.88 & 0.93 & 9.05 \\
avazu-site
& $16{,}000$ & $1{,}000{,}000$
& 0.78 & \textbf{0.52} & 1.11 & 9.24 \\
criteo
& $16{,}000$ & $1{,}000{,}000$
& 0.75 & \textbf{0.38} & 0.88 & 9.09 \\
duke breast-cancer
& $35$ & $7{,}129$
& 0.047 & \textbf{0.030} & 0.055 & 0.22 \\
gisette
& $4{,}800$ & $5{,}000$
& 3.86 & 2.43 & 3.98 & \textbf{2.30} \\
kdd2010 (algebra)
& $16{,}000$ & $20{,}216{,}830$
& 33.6 & \textbf{22.1} & 39.9 & 379.9 \\
kdd2010 (bridge-to-algebra)
& $16{,}000$ & $29{,}890{,}095$
& 53.3 & \textbf{29.3} & 61.9 & 574.4 \\
kdd2010 raw (bridge-to-algebra)
& $16{,}000$ & $1{,}163{,}024$
& 0.79 & \textbf{0.75} & 1.12 & 11.2 \\
kdd2012
& $16{,}000$ & $54{,}686{,}452$
& 57.0 & \textbf{33.2} & 78.7 & 807.5 \\
leukemia
& $31$ & $7{,}129$
& 0.038 & \textbf{0.036} & 0.051 & 0.23 \\
news20.binary
& $15{,}997$ & $1{,}355{,}191$
& 3.23 & \textbf{1.22} & 3.95 & 22.9 \\
rcv1.binary
& $16{,}000$ & $47{,}236$
& 0.25 & \textbf{0.16} & 0.29 & 0.97 \\
real-sim
& $16{,}000$ & $20{,}958$
& 0.17 & \textbf{0.12} & 0.20 & 0.45 \\
url
& $16{,}000$ & $3{,}231{,}961$
& 4.09 & 5.77 & \textbf{3.38} & 33.5 \\
webspam
& $16{,}000$ & $16{,}609{,}143$
& 42.0 & \textbf{26.0} & 48.5 & 201.0 \\
\bottomrule
\end{tabular}
\end{table}

All four methods reach the stopping threshold in every run on all $15$ data sets.  The proposed methods achieve the lowest median first-hit time on $13$ data sets, with NR-LALM+SOC being fastest on $12$.  Each proposed method is faster than L-AL on $14$ data sets, with geometric-mean speedups of $1.19$ for NR-LALM and $1.71$ for NR-LALM+SOC.  On the four largest problems, whose dimensions range from $16.6$ million to $54.7$ million, NR-LALM+SOC is the fastest method, with speedup factors ranging from $1.81$ to $2.37$ over L-AL and from $7.73$ to $24.3$ over IPOPT.  Across all $15$ data sets, the geometric means of the IPOPT-to-NR-LALM and IPOPT-to-NR-LALM+SOC first-hit-time ratios are $6.64$ and $9.48$, respectively.

\subsection{Stochastic performance under matched gradient budgets}
\label{subsec:stochastic-comparison}

Following \cite{BerahasBollapragadaGupta2025}, we consider stochastic multiclass logistic regression on the LIBSVM covtype and MNIST data sets~\cite{ChangLin2011}. For observations \((a_j,y_j)\), class vectors \(w_i\in\R^{n_f}\), and
\(x=(w_1,\ldots,w_{K_c})\in\R^{K_c n_f}\), the problem is
\begin{equation}
\begin{aligned}
 \min_x\quad
 f(x)&:=\frac1N\sum_{j=1}^N
 \left[\log\!\left(\sum_{i=1}^{K_c}\exp(a_j^\top w_i)\right)
       -a_j^\top w_{y_j}\right],\\
 \text{subject to}\quad
 c_i(x)&:=\frac{\norm{w_i}^2-1}{2}=0,
 \qquad i=1,\ldots,K_c.
\end{aligned}
 \label{eq:stochastic-multiclass-sphere}
\end{equation}
Only objective component gradients are sampled; the constraints and their Jacobians are evaluated exactly.  With a bias feature, \((N,n,m)\) equals
\((581{,}012,385,7)\) for covtype and \((60{,}000,7{,}810,10)\) for MNIST.
All methods use the same preprocessed initial point and \(\lambda_0=0\), with
preprocessing excluded from the algorithmic budgets.

We compare NR-LALM, NR-LALM+SOC, MLALM~\cite{ShiWangWang2025}, and stochastic
sequential quadratic programming (S-SQP)~\cite{BerahasEtAl2021}.  To avoid
dependence on the returned multiplier, we use the full-data residual
$$
\mathcal R_{\min}^2(x)
:=\min_{\lambda\in\R^m}
\left\{\norm{\grad f(x)+\grad c(x)\lambda}^2+\norm{c(x)}^2\right\}.
$$
It gives a common primal-point comparison, but
\(\mathcal R_{\min}^2(x)\le\mathcal R(x,\lambda)^2\) for every \(\lambda\), so
it does not assess the multiplier returned by a method. The proposed methods share samples and use projected SPIDER with
\((\rho,\beta)=(300,48)\), selected on separate development runs and fixed for
both data sets; the baselines use the parameterizations in their cited
references.  Each method receives \(2^{18}\) stochastic objective-gradient
evaluations per run.  Figure~\ref{fig:stochastic-kkt-residual} reports
unsmoothed arithmetic means over ten runs; common full-data evaluations are
excluded from this budget.  Complete settings are provided in the public code.

\begingroup
\setlength{\intextsep}{4pt}
\begin{figure}[H]
\centering
\includegraphics[width=0.50\textwidth]
{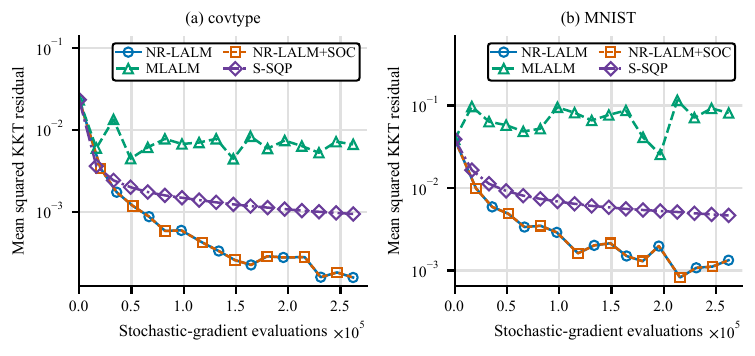}
\caption{Mean \(\mathcal R_{\min}^2\) over ten runs on \textup{(a)} covtype and \textup{(b)} MNIST.}
\label{fig:stochastic-kkt-residual}
\end{figure}
\endgroup

Both proposed methods outperform both baselines at all \(16\) post-initial checkpoints. Their final mean \(\mathcal R_{\min}^2\) values on covtype and MNIST are \(1.95\times10^{-4}\) and \(1.27\times10^{-3}\), versus \(9.22\times10^{-4}\) and \(4.66\times10^{-3}\) for the best baseline, reductions by factors of about \(4.7\) and \(3.7\). The proposed curves nearly coincide. Final mean feasibility is below \(10^{-6}\), whereas least-squares stationarity is \(1.37\times10^{-2}\) and \(3.48\times10^{-2}\); hence stationarity dominates \(\mathcal R_{\min}^2\). Figure~\ref{fig:mechanism-verification}\textup{(b)} isolates SOC's higher-order constraint effect.

\FloatBarrier

\section{Conclusion}
\label{sec:conclusion}

We introduced NR-LALM, combining a linearized primal step with the classical nonlinear-residual multiplier update. Our analysis controls the second-order constraint-linearization error and establishes bounded multipliers, trajectory localization, and $\mathcal O(\varepsilon^{-2})$ deterministic iteration and first-order-oracle complexity with fixed parameters. With safeguarded restarts, projected SPIDER achieves expected $\mathcal O(\varepsilon^{-3})$ stochastic-gradient and $\mathcal O(\varepsilon^{-2})$ constraint/Jacobian evaluation complexities under local regularity. The optional SOC reduces the error from second to fourth order without changing these orders; compactness and the KL property give finite-length convergence of the deterministic primal--dual sequence to a KKT pair. Numerics confirm the predicted orders and favorable performance. The Lean formalization and public code support independent checking of the theory and experiments. Inequalities, inexact solves, adaptive parameters, and globalization remain open.

\FloatBarrier
\vspace{-0.5\baselineskip}

\section*{Acknowledgments}

Generative AI assisted manuscript preparation and parts of the mathematical and computational work. The authors verified all results and assume responsibility for all content.

\bibliographystyle{abbrv}
\bibliography{references}

\end{document}